\documentclass[12pt,a4paper,oneside,draft,reqno]{amsart}

\usepackage{amsmath} % Lots of maths functionality
\usepackage{amssymb} % Maths symbols
\usepackage{mathrsfs} % Maths scripted font.
\usepackage{amsthm} % Maths environments: \begin{proof}, etc.
\usepackage[english]{babel} % Language and hyphenation.
\usepackage[T1]{fontenc} % For font encoding, to allow accents, copy & paste, inequality signs, etc. to all work nicely.
\usepackage[utf8]{inputenc} % To be loaded after fontenc, also for encoding.
\usepackage{mathtools} % Uses amsmath, fixes quirks and adds functionality.
\usepackage[pdftex,  colorlinks=true]{hyperref} % Makes references and citations into links
\usepackage{setspace} % Allows \onehalfspacing etc. for changing gaps between lines
\usepackage{tikz-cd} % Commutative diagrams
\usepackage[capitalize]{cleveref} % use \Cref{} for instead of X~\ref{} 

\numberwithin{equation}{section}

\makeatletter
\@namedef{subjclassname@1991}{Mathematical subject classification 1991}
\@namedef{subjclassname@2000}{Mathematical subject classification 2000}
\@namedef{subjclassname@2010}{Mathematical subject classification 2010}
\@namedef{subjclassname@2020}{Mathematical subject classification 2020}
\makeatother

\newtheorem{theorem}{Theorem}[section]
\newtheorem{lemma}[theorem]{Lemma}
\newtheorem{corollary}[theorem]{Corollary}
\newtheorem{proposition}[theorem]{Proposition}

\newtheorem{thmx}{Theorem}

\newtheorem{meta-question}[mquex]{Meta-question}

\theoremstyle{definition}
\newtheorem{definition}[theorem]{Definition}
\newtheorem{remark}[theorem]{Remark}

\newtheorem*{examples*}{Examples}

\newtheoremstyle{TheoremNum}
{\topsep}{\topsep} %%% space between body and thm
{\itshape} %%% Thm body font
{-0.25cm} %%% Indent amount (empty = no indent)
{\bfseries} %%% Thm head font
{.} %%% Punctuation after thm head
{ }  %%% Space after thm head
{\thmname{#1}\thmnote{ \bfseries #3}}%%% Thm head spec
\theoremstyle{TheoremNum}

\DeclareMathOperator{\aut}{\mathrm{Aut}}

\DeclareMathOperator{\coker}{\mathrm{coker}}

\DeclareMathOperator{\Ext}{\mathrm{Ext}}

\DeclareMathOperator{\Tr}{\mathrm{Tr}}
\DeclareMathOperator{\im}{\mathrm{im}}
\DeclareMathOperator{\id}{\mathrm{Id}}
\DeclareMathOperator{\fix}{\mathrm{Fix}}

\DeclareMathOperator{\isom}{\mathrm{Isom}}

\DeclareMathOperator{\Sym}{Sym}

\newcommand{\GL}{\mathrm{GL}}

\newcommand{\SL}{\mathrm{SL}}
\newcommand{\PSL}{\mathrm{PSL}}

\newcommand{\gen}[1]{\langle #1 \rangle}
\newcommand{\pres}[2]{\langle #1\ |\ #2 \rangle}

\DeclareMathOperator{\lcm}{\mathrm{lcm}}

\newcommand{\N}{\mathbb{N}}
\newcommand{\Z}{\mathbb{Z}}

\newcommand{\R}{\mathbb{R}}

\usepackage{tikz}
\usetikzlibrary{arrows,quotes}
\tikzstyle{blackNode}=[fill=black, draw=black, shape=circle]

\newcommand{\Iff}{if and only if}
\newcommand{\fg}{finitely generated}
\newcommand{\gp}{group}
\newcommand{\gps}{groups}
\newcommand{\sgp}{subgroup}
\newcommand{\sgps}{subgroups}
\newcommand{\st}{such that}
\newcommand{\rf}{residually-finite}
\newcommand{\wrt}{with respect to}
\newcommand{\actson}{\curvearrowright}
\newcommand{\sub}{\subset}
\renewcommand{\i}{^{-1}}

\usepackage[foot]{amsaddr}
\title[Profinite rigidity properties of central $\Z^n$-by-($2$-orbifold)]{Profinite rigidity properties of central extensions of 2-orbifold groups}
\author{Pawe\l\ Piwek}
\date{March 2023}
\address{Mathematical Institute, Andrew Wiles Building, Observatory Quarter, University of Oxford, Oxford OX2 6GG, UK}
\email{pawel.piwek@maths.ox.ac.uk}

\begin{document}
	
\newpage
	
\maketitle

\begin{abstract}
	We extend Wilkes' results on the profinite rigidity of SFSs
	to the setting of central extensions of 2-orbifold groups with higher-rank centre.
	We prove that both rigid and non-rigid phenomena arise in this setting
	and that the non-rigid phenomena are transient in the sense that
	if $\widehat{G}_1 \cong \widehat{G}_2$, then $G_1 \times \Z \cong G_2 \times \Z$.
\end{abstract}

\section{Introduction}
\label{sec:introduction}

One of the central pursuits of group theory has always been
the search for ways of distinguishing between
non-isomorphic groups.
An invariant that is important
from both a theoretical and computational viewpoint is
the set $\mathcal{C}(G)$ of (isomorphism types of)
all finite quotients of a given group $G$.
Studying this invariant it quickly becomes obvious
that it is wise to assume the \emph{residual-finiteness} of $G$
-- i.e. that every non-trivial element of $G$
is non-trivial in some finite quotient.

One can also ask whether or not
considering just the \emph{set} $\mathcal{C}(G)$ of finite quotients
is a good enough invariant, or whether the algebraic structure
of homomorphisms between them should be `remembered' somehow.
This leads to the definition of
the \emph{profinite completion} $\widehat G$ of a group $G$.
Most remarkably, it turns out that
for finitely generated groups $G$ and $H$
their completions $\widehat G$~and~$\widehat H$ are isomorphic
\Iff\ $\mathcal{C}(G) = \mathcal{C}(H)$.

As with any invariant in mathematics,
it is natural to ask: which objects does it distinguish (uniquely)?

\begin{meta-question}\label{qu:original}
	Given a \fg\ residually finite group $G$, is there a \fg\ residually finite group $H$, such that $\widehat G\cong \widehat H$, but $G \not \cong H$?
\end{meta-question}

It isn't difficult to give $G$ for which the answer is positive.
Baumslag gave examples of two non-isomorphic semidirect products $(\Z/25)\rtimes \Z$,
which share profinite completions
-- see \cite{baumslag1974residually}.
Also, related to our article, Hempel in \cite{hempel2014some}
gave examples of mapping tori for self-homeomorphisms of surfaces
which produced non-homeomorphic $3$-manifolds
with isomorphic profinite completions.

On the other hand, it wasn't until \cite{bridson2020absolute}
that we got `full-sized' examples of \emph{profinitely rigid} groups
-- groups which don't share their profinite completion
with any other residually finite finitely generated group.
This is indicative of the fact that
\cref{qu:original} is very difficult in general
which prompts modifying it to the following \cref{qu:within_class}.

\begin{meta-question}\label{qu:within_class}
	Given a residually finite group $G$, is there a residually finite group $H$ \emph{within a specific class of groups} $\mathscr{D}$, such that $\widehat G\cong \widehat H$, but $G \not \cong H$?
\end{meta-question}

\Cref{qu:within_class} is often much more approachable
because one can use structure theorems
for groups in the class $\mathscr{D}$.
For example \cite{bridson2013determining}
showed that Fuchsian groups are distinguished from each other
(and other lattices in connected Lie groups)
by their profinite completions.
Wilton showed in \cite{wilton2018essential,wilton2021profinite}
that every finitely generated free group and surface group
is distinguished by profinite completion
from all other limit groups.
This was recently extended by Morales
in \cite{morales2022profinite}
to the class of finitely generated residually free groups.

Significant advances have been made in the study of
profinite rigidity properties of $3$-manifold groups --
for a survey see Reid's article \cite{reid2018profinite}.
Among others, Funar---using classical work of Stebe---showed in \cite{funar2013torus}
that torus bundles with Sol geometry are an ample source of
non-isomorphic pairs of groups with isomorphic profinite completions;
Wilton and Zalesskii proved in \cite{wilton2017distinguishing} that
profinite completions distinguish the eight geometries
of closed orientable $3$-manifolds with infinite fundamental groups.

Finally, particularly relevant to this article,
Wilkes showed in \cite{wilkes2017profinite}
that the fundamental groups of Seifert Fibre Spaces
are distinguished from each other apart from the examples given by Hempel.
The present work aims at extending a part of this result
by considering group extensions with centre $\Z^n$ for $n > 1$
and the quotient being a $2$-orbifold group.

\subsection*{Original results}

The first result of this paper reduces the study of profinite rigidity
within the class of
central extensions of free abelian groups
by infinite closed orientable $2$-orbifold groups
to the situation where the kernel and the quotient are fixed.

A similar result is contained in \cite{bridson2023profinite}.

\begin{thmx}
	Let $n_1, n_2$ be natural numbers
	and $\Delta_1, \Delta_2$ be infinite
	fundamental groups
	of closed orientable $2$-orbifolds.
	Let $G_1$ and $G_2$ be central extensions
	$\Z^{n_1}$-by-$\Delta_1$
	and $\Z^{n_2}$-by-$\Delta_2$ respectively.
	If ${\widehat G_1 \cong \widehat G_2}$
	then $n_1 = n_2$ and $\Delta_1 \cong \Delta_2$.
\end{thmx}

The main result is \cref{thm:main}
showing that given an infinite closed orientable $2$-orbifold group $\Delta$
the central extensions with kernel $\Z^n$ and quotient $\Delta$
are distinguished from each other by their profinite completions
for large enough $n$.
Also, for most such groups $\Delta$,
depending on the order of cone points,
the result specifies exactly when the rigidity happens
and when it doesn't.

\begin{thmx}\label{thm:main}
	Let $\Delta$ be an infinite fundamental group
	of a closed orientable {$2$-orbifold}
	with $m \ge 0$ cone points
	of orders $p_1, \ldots, p_m$.
	Furthermore, for $j = 1, \ldots, m$ set
	$$d_j = \frac{\gcd(p_{i_1}p_{i_2}\ldots p_{i_j}\
		\text{for}\ 1\le i_1 < i_2 < \ldots < i_j \le m)}
	{\gcd(p_{i_1}p_{i_2}\ldots p_{i_{j-1}}\
		\text{for}\ 1\le i_1 < i_2 < \ldots < i_{j-1} \le m)}.$$
	Then, for $n > 1$ the following hold.
	\begin{enumerate}
		\item If $n > m$,
		or $d_{m-(n-1)} \in \{1, 2, 3, 4, 6\}$,
		then the non-isomorphic central extensions
		of $\Z^n$ by $\Delta$
		are distinguished from each other by their profinite completions.
		\item If $n \le m$
		and $d_{m-(n-1)} \not\in \{1, 2, 3, 4, 6, 12\}$,
		then there exist non-isomorphic central extensions $G_1, G_2$
		of $\Z^n$ by $\Delta$
		with $\widehat G_1 \cong \widehat G_2$.
	\end{enumerate}
\end{thmx}

The third result shows that in fact all examples of lack of rigidity
come from the phenomenon described by Baumslag in \cite{baumslag1974residually}
-- that ${\widehat G_1 \cong \widehat G_2}$ implies in this context
that $G_1 \times \Z \cong G_2 \times \Z$.

\begin{thmx}\label{thm:main_appendix}
	Let $\Delta$ be an infinite fundamental group
	of a closed orientable $2$-orbifold,
	with $m \ge 0$ cone points
	of orders $p_1, \ldots, p_m$.
	Let $n>1$ and $G_1, G_2$ be central extensions
	of $\Z^n$ by $\Delta$
	such that $\widehat G_1 \cong \widehat G_2$.
	Then $G_1 \times \Z \cong G_2 \times \Z$.
\end{thmx}

Note that in \cref{thm:main} and \cref{thm:main_appendix}
we made the assumption that $n > 1$.
This assumption isn't necessary,
but it simplifies the statements.
The case of $n=1$ has already been covered
by \cite{hempel2014some, wilkes2017profinite}.

\subsection*{Structure of this article}

\Cref{sec:introduction} states the main questions of this paper
and gives the related results as context for them;
it states the original results
and discusses the structure of this article.

\Cref{sec:background} contains the necessary technical background
and minor lemmas:
\cref{ssec:profinite_completions} introduces profinite completions
and the implications of goodness to residual-finiteness;
\cref{ssec:2_orbifolds} discusses 2-orbifold groups, their automorphisms and centres;
\cref{ssec:cohomology} introduces the needed results of group cohomology
and computes the cohomology groups.

In particular,
in \cref{sssec:rf_of_our_extensions} we show
that the extensions in question are residually-finite
and that the topology induced on the kernel
is the full profinite topology.
In \cref{sssec:aut_torsion_orbifolds} we quote
important results on automorphisms of $2$-orbifold groups
and their profinite completions,
specifying how they act on the maximal torsion elements.
In \cref{sssec:centres} we show that the $2$-orbifold groups we consider
and their profinite completions
are centreless.
\Cref{sssec:Hopfs_formula} gives an explicit form
of the standard Cohomological Hopf's Formula,
which we later need for calculations.
Finally, \Cref{sssec:cohomology_computed}
computes the cohomology groups of the $2$-orbifold groups we consider
and also their profinite completions,
as well as the actions
of their automorphism groups on the cohomology groups themselves.

\Cref{sec:matrix_correspondence} translates the questions
of isomorphism of central extensions of $\Z^n$ by 2-orbifold groups
and isomorphism of their profinite completions
to questions
about orbits of certain actions by matrices on sets of matrices.

\Cref{sec:matrix_correspondence_implications} explores the translated question
and gives a thorough classification
of when distinct orbits of the first action
collapse to one orbit under the second action.

Finally, \cref{sec:results} states the final results
back in the setting of
distinguishing non-isomorphic central extensions
of $\Z^n$ by $2$-orbifold groups
and proves them.

\subsection*{Acknowledgements}

The author is thankful to his PhD supervisor,
Martin Bridson,
for an introduction to the problem, general guidance
and numerous helpful conversations.

This work was supported by the Mathematical Institute Scholarship
of University of Oxford.

\section{Background}\label{sec:background}

\subsection{Profinite completions}\label{ssec:profinite_completions}

Since this article isn't intended as a thorough introduction
to the subject of profinite groups and profinite completions
only the very basic definitions are given here.
The reader is referred to the lecture notes of Wilkes \cite{wilkes2021profinitegroups}
for a proper introduction to the topic
and to the book of Ribes and Zalesskii \cite{ribes2000profinite}
for a comprehensive reference.

\subsubsection{Profinite completions}

The \emph{profinite completion} of a group
organises `the set of all finite quotients of a group' into an algebraic object
which also `remembers' how the different quotients relate to each other.

For each group $G$ we can form an inverse system
$\mathcal{N} = \{G/N\ |\ N \triangleleft_{\textrm{fi}} G\}$ of its canonical finite quotients.
It is indeed an inverse system
since for any finite index normal subgroups $N < M$ of $G$
we have a canonical map $\varphi_{NM}\colon  G/N \to G/M$
and given any $N, M \triangleleft_{\textrm{fi}} G$,
the intersection $N \cap M$ is also a finite index normal subgroup of $G$.

\begin{definition}
	Given a group $G$ together with the inverse system of its canonical finite quotients $G/N$,
	we define its \emph{profinite completion} to be the inverse limit
	$$\widehat G := \varprojlim_{N \triangleleft_{\mathrm{fi}} G} G/N.$$
\end{definition}

The limit $\widehat G$ may be realised as a subgroup
of the direct product $\Pi_{N\triangleleft_{\mathrm{fi}} G}\ G/N$
which may be treated as a compact space
by giving the finite groups $G/N$ the discrete topology.
This makes $\widehat G$ into a compact Hausdorff topological group.

From the canonical maps $\psi_N\colon G \to G/N$
we get a canonical map ${\mathsf{h}\colon  G \to \widehat G}$
which is injective \Iff\ the group $G$ is \rf,
i.e. the intersection of its finite index subgroups is the trivial subgroup.
The image $\mathsf{h}(G)$ is dense in $\widehat G$.

The following striking result has become standard
-- see \cite{dixon1982profinite} for the original paper
or \cite[Corollary 3.2.8.]{ribes2000profinite} for a more general version.
We write $\mathcal{C}(G)$ for the set of isomorphism classes of finite quotients of $G$.

\begin{theorem}\label{thm:completion_isomorphism_sets_of_quotients}
	Let $G$ and $H$ be finitely generated groups.
	Then we have ${\mathcal{C}(G) =  \mathcal{C}(H)}$
	\Iff\ $\widehat G \cong \widehat H$.
\end{theorem}

\subsubsection{Profinite completions and extensions of groups}

We will need to know what happens to an extension of groups under profinite completion.
The following is taken from \cite[Proposition 3.2.5.]{ribes2000profinite}.
Here and throughout $\overline H$ denotes the closure
of $\mathsf{h}(H)$ in $\widehat G$ for $H < G$.

\begin{proposition}\label{prop:SES_completion}
	Given a short exact sequence $1\to N \to G\to Q \to 1$ of any \gps,
	the following diagram commutes.
	\begin{equation*}
		\begin{tikzcd}
			1\ar[r] &
			N\ar[r]\arrow{d}{\mathsf{h}_G|_N} &	G\arrow{r}{\pi}\arrow{d}{\mathsf{h}_G} &	Q\ar[r]\arrow{d}{\mathsf{h}_Q} &
			1
			\\
			1\ar[r]	&
			\overline N\ar[r] &
			\widehat{G}\arrow{r}{\hat \pi} &
			\widehat{Q}\ar[r]	&
			1
		\end{tikzcd}
	\end{equation*}
\end{proposition}

\subsubsection{Residual finiteness and goodness}

In this section we give a technical proposition,
which will then be instrumental for proving that the groups we consider in this article are \rf.
It is centered on the notion of \emph{goodness} introduced by Serre in \cite{serre1979galois}
which we define later in \cref{sssec:profinite_cohomology} in \cref{def:goodness}.
The result is adapted from \cite[Corollary 6.2.]{grunewald2008cohomological}.

\begin{proposition}\label{cor:rf-by-good}
	Let $N$ be a residually finite, \fg\ \gp\
	and $Q$ be a good \rf\ \gp.
	Then any extension of $N$ by $Q$ is residually finite
	and induces the full profinite topology on $N$.
\end{proposition}

\subsection{$2$-orbifold groups}\label{ssec:2_orbifolds}
\subsubsection{$2$-orbifold groups}

When a group $G$ acts on a $n$-manifold $M$ properly discontinuously and freely,
the quotient space $G \backslash M$ is naturally a manifold
and we get a covering $M \twoheadrightarrow G \backslash M$.

\emph{Orbifolds} are objects that describe the quotient accurately
when the action has non-trivial finite isotropy groups.

For us it is natural to set aside
the technicalities associated to orbifolds
and consider only their fundamental groups.
Moreover, we restrict ourselves to only the geometric orbifolds.

\begin{definition}
	Let $M$ be one of $\mathbb{R}^n, \mathbb{H}^n,$ or $\mathbb{S}^n$.
	A group $G < \isom(M)$ is called an \emph{$n$-orbifold group}
	if its action on $M$ is properly discontinuous.
	We say that an $n$-orbifold \gp\ $G$ is \emph{orientable} if $G < \isom^+(M)$.
	We call a $2$-orbifold group \emph{closed}
	if the quotient topological space $G\backslash  M$ is homeomorphic to a closed surface $S$.
\end{definition}

\begin{examples*}\phantom{0}
	\begin{enumerate}
		\item Any surface group $\Sigma$ is a closed $2$-orbifold group.
		\item $\PSL_2 (\Z) < \isom^+(\mathbb{H}^2)$ is a $2$-orbifold group,
		but it isn't a \emph{closed} $2$-orbifold group,
		as the quotient $\PSL_2(\Z) \backslash \mathbb{H}^2$ is a punctured sphere.
		\item Any \emph{branched cover} of surfaces of degree $d$
		gives rise to a pair $\Sigma < \Delta$
		where $\Sigma$ is the fundamental group of the covering surface,
		$\Delta$ is a $2$-orbifold group
		(which is in general \emph{not}
		the fundamental group of the underlying surface of the quotient)
		and $[\Delta: \Sigma] = d$.
		\item Every discrete lattice in $\PSL_2 (\R)$ (i.e. a Fuchsian group) is a $2$-orbifold group,
		but only the \emph{cocompact} lattices are closed $2$-orbifold groups. 
	\end{enumerate}
\end{examples*}

The closed orientable $2$-orbifold groups have presentations which are convenient to work with.
A proof of \cref{prop:2-orbs_presentation} relies on a version of Seifert-van Kampen theorem for orbifolds
and can be found in \cite[Sections 4.7. and 5.1.]{choi2012geometric}.

\begin{proposition}\label{prop:2-orbs_presentation}
	Any closed orientable $2$-orbifold group $G$ has a presentation
	\begin{equation}\label{eq:presentation}
		\begin{aligned}
			\langle x_1, \ldots, y_g, a_1, \ldots, a_m\
			|\ & \Pi_{i=1}^{g} [x_i, y_i]\cdot a_1\ldots a_m = 1,\ \\
			&a_i^{p_i} = 1 \text{ for } i = 1, \ldots, m\rangle,
		\end{aligned}
	\end{equation}
	where $g$ is the genus of the quotient surface $G \backslash M$,
	while $m$ is the number of orbits of singular points
	and $p_i$ is the size of the stabilizers of the $i$-th orbit.
\end{proposition}

To save some space in later sections,
we introduce the following definition.
The reader should be forewarned
that the term \emph{nice} $2$-orbifold group is by no means standard. 

\begin{definition}\label{def:nice_group}
	A $2$-orbifold group is called \emph{nice}
	if it is closed and orientable,
	infinite and not isomorphic to $\Z^2$.
\end{definition}

It turns out that almost all of the closed orientable $2$-orbifold groups are nice.

\begin{proposition}\label{prop:which_orbs_are_nice}
	The following are equivalent for a closed orientable $2$-orbifold group $G < \isom^+(M)$
	with an underlying surface of genus $g$
	and $m \ge 0$ cone points of orders $p_1, \ldots, p_m$.
	\begin{itemize}
		\item $G$ is nice.
		\item One of the following conditions holds.
		\begin{itemize}
			\item $g > 1$.
			\item $g = 1$, and $m \ge 1$.
			\item $g = 0$, and $m \ge 4$.
			\item $g = 0$, and $m = 3$ and $\frac{1}{p_1} + \frac{1}{p_2} + \frac{1}{p_3} \le 1$.
		\end{itemize}
	\end{itemize}
\end{proposition}

\begin{proof}
	If $G$ is infinite if and only if
	the orbifold Euler characteristic is non-positive,
	which is
	$\chi = 2 - 2g - \sum_{i = 1}^m(1 - \frac{1}{p_i})$.
	The only case that needs to be excluded is $G \cong \Z^2$
	which is $g = 1$ and $m = 0$.
\end{proof}

We will later use the following fact,
being a consequence of \cite[Theorem 2.5.]{scott1983geometries},
which tells us that the family of $2$-orbifold groups
isn't too far from the family of surface groups.

\begin{proposition}\label{prop:orbifolds_have_manifolds}
	Let $\Delta$ be a $2$-orbifold group.
	Then, there is a finite index subgroup $\Sigma < \Delta$
	isomorphic to a surface group.
\end{proposition}

Also, it will be useful to know the torsion elements of nice $2$-orbifold groups.

\begin{proposition}\label{prop:torsion_elts_in_Delta}
	Let $\Delta$ be a nice $2$-orbifold group with presentation
	\begin{equation*}
		\begin{aligned}
			\langle x_1, \ldots, y_g, a_1, \ldots, a_m\
			|\ & \Pi [x_i, y_i]\cdot a_1\ldots a_m = 1,\ \\
			&a_i^{p_i} = 1 \text{ for } i = 1, \ldots, m\rangle.
		\end{aligned}
	\end{equation*}
	Then any torsion element of $\Delta$
	is conjugate to a power of one of $a_i$.
\end{proposition}

\begin{proof}
	As an oriented $2$-orbifold group, $\Delta$ is a subgroup of $\isom^+(M)$
	and since it's infinite $M = \R^2$ or $\mathbb{H}^2$.
	Any torsion element of $\isom^+(M)$ must fix a point in $M$.
	Since the action away from the lifts of cone points is free,
	the fixed point of the torsion element must be
	one of the lifts of cone points
	and thus the torsion element
	is conjugate to a power of one of $a_i$.
\end{proof}

\subsubsection{Residual finiteness of central extensions of $2$-orbifold groups}\label{sssec:rf_of_our_extensions}

An important consequence of \cref{prop:orbifolds_have_manifolds}
is that $2$-orbifold groups are residually finite.

\begin{proposition}\label{prop:2-orbs_rf}
	All $2$-orbifold groups are residually finite.
\end{proposition}

\begin{proof}
	By \cref{prop:orbifolds_have_manifolds}
	$2$-orbifold groups have finite index surface subgroups
	and these are residually finite by \cite{hempel1972residual}.
	A finite index supergroup of a \fg\ residually finite group is residually finite itself.
\end{proof}

We will also need to know that the $2$-orbifold groups in question are good.

\begin{proposition}\label{prop:2-orbs_good}
	Closed orientable $2$-orbifold groups
	are good in the sense of Serre.
\end{proposition}

\begin{proof}
	\cite{grunewald2008cohomological} defines a class of \emph{$\mathscr{F}$-groups}
	to be all groups with a presentation of the form
	\begin{equation*}
		\begin{aligned}
			\pres{a_1, \ldots, b_g, c_1, \ldots, c_t, d_1, \ldots, d_s}%
			{& c_1^{e_1}=\ldots=c_t^{e_t}=1,\\& \Pi[a_i, b_i] = c_1\ldots c_t d_1\ldots d_s},
		\end{aligned}
	\end{equation*}
	and \cite[Proposition 3.6.]{grunewald2008cohomological}
	shows that all $\mathscr{F}$-groups are good.
	By \cref{prop:2-orbs_presentation},
	closed orientable $2$-orbifold groups are $\mathscr{F}$-groups.
\end{proof}

As a consequence of \cref{cor:rf-by-good}
we get \cref{cor:central_exts_are_rf}.

\begin{corollary}\label{cor:central_exts_are_rf}
	Let $1 \to N \to G \to Q \to 1$ be a short exact sequence
	such that $N = \Z^n$ for some $n$
	and $Q$ is a closed orientable $2$-orbifold group.
	Then $G$ is \rf\ and $\overline N = \widehat N$ in $\widehat G$.
\end{corollary}

%\begin{proof}
%	We will use \cref{cor:rf-by-good}.
%	Take a cofinal sequence $N_i \triangleleft N$ of characteristic finite index subgroups of $N$.
%	Then $1 \to N/N_i \to G/N_i \to Q \to 1$ is a SES
%	and $N/N_i$ is a \fg\ \rf\ group,
%	so $G/N_i$ is \rf,
%	and so $N_i$ is an intersection of finite index subgroups of $G$ which contain it.
%	Since the sequence $N_i$ is cofinal,
%	this shows that $G$ is \rf,
%	and also that $\overline N = \widehat N$.
%\end{proof}

\subsubsection{Automorphisms and torsion elements of $\Delta$ and $\widehat \Delta$}\label{sssec:aut_torsion_orbifolds}

It will be important for us to study the automorphisms
of nice $2$-orbifold groups and their profinite completions.
A result of \cite[Theorem 5.1.]{bridson2013determining},
which we state as \cref{thm:automorphism_cone_points_Fuchsian},
makes it viable for Fuchsian groups.
To include the nice $2$-orbifold groups which aren't Fuchsian groups,
we rephrase \cite[Proposition 4.3.]{wilkes2017profinite}
as \cref{prop:automorphism_cone_points_Euclidean}.

\begin{theorem}\label{thm:automorphism_cone_points_Fuchsian}
	Let $G$ be a \fg\ Fuchsian \gp,
	then every finite \sgp\ of $\widehat{G}$ is conjugate to a \sgp\ of $G$
	and if two maximal finite \sgps\ of $G$ are conjugate in $\widehat{G}$,
	then they are already conjugate in $G$.
\end{theorem}

\begin{proposition}\label{prop:automorphism_cone_points_Euclidean}
	Let $G$ be a nice $2$-orbifold group
	isomorphic to a subgroup of $\isom^+(\R^2)$.
	Then every torsion element of $\widehat{G}$ is conjugate to a torsion element of $G$,
	and if two torsion elements of $G$ are conjugate in $\widehat{G}$,
	then they are already conjugate in $G$.
\end{proposition}

\Cref{thm:automorphism_cone_points_Fuchsian}
and \cref{prop:automorphism_cone_points_Euclidean}
together with \cref{prop:torsion_elts_in_Delta}
show that the generators $a_i$ of a nice $2$-orbifold group $G$
(as in presentation~\ref{eq:presentation})
are sent under any automorphism $\phi\colon  \Delta \to \Delta$
(respectively: under any automorphism $\phi\colon \widehat \Delta \to \widehat \Delta$) to
$$^{g_i}(a_{\sigma(i)}^{k_i}) := g_i a_{\sigma(i)}^{k_i} g_i^{-1}$$
where $g_i \in \Delta$ (respectively: $g_i \in \widehat\Delta$)
and $\sigma$ is a permutation in $\Sym(m)$.

Finally, it will prove useful to show that
any permutation of $a_i$ which respects the orders $p_i$
is actually possible.

\begin{proposition}\label{prop:any_permutation_is_induced}
	Let $\sigma\in \Sym(m)$ be a permutation of $\{1, 2, \ldots, m\}$
	such that $p_i = p_{\sigma(i)}$ for all $i=1, \ldots, m$.
	Then there exists an automorphism $\phi \in \aut(\Delta)$
	\st\ $\phi(a_i) =\ ^{g_i}(a_{\sigma(i)})$ for some $g_i \in \Delta$.
\end{proposition}

\begin{proof}
	It is enough to show that for $\sigma$ being a single transposition,
	say $\sigma = (ij)$.
	Then
	\begin{align*}
		&a_1 \ldots a_i \ldots a_j \ldots a_m \\
		&= a_1 \ldots a_{i-1} \cdot\ ^{a_i}\big(a_{i+1} \ldots a_{j}\big)\cdot a_i \cdot a_{j+1} \ldots a_m \\
		&= a_1 \ldots a_{i-1} \cdot\ ^{a_i}a_j \cdot\ ^{a_ia_j\i}\big(a_{i+1}\ldots a_{j-1}\big) \cdot a_i \cdot a_{j+1}\ldots a_m
	\end{align*}
	so we can define a homomorphism of $\Delta$
	by setting the following images of generators $a_i$
	(and keeping generators $x_i, y_i$ fixed)
	\begin{align*}
		a_1 & \mapsto a_1 & a_{i+1} & \mapsto\ ^{a_ia_j\i}(a_{i+1}) & 	a_{j+1} & \mapsto a_{j+1} \\
		&\vdots & &\vdots  & &\vdots \\
		a_{i-1} & \mapsto a_{i-1} & a_{j-1} & \mapsto\ ^{a_ia_j\i}(a_{j-1}) & a_m & \mapsto a_m \\
		a_i & \mapsto\ ^{a_i}a_j & a_j & \mapsto a_i
	\end{align*}
	
	All of the relators are sent to relators
	and each of the generators is in the image, so the map is surjective.
	Since $\Delta$ is \fg\ \rf, it is Hopfian and so $\phi$ is an automorphism.
\end{proof}

\subsubsection{Trivial centres of $\Delta$ and $\widehat{\Delta}$}\label{sssec:centres}

The fact that nice $2$-orbifolds have trivial centres isn't difficult to prove.
It is more difficult though to extend this to the profinite completion,
which is the content of \cref{thm:completions_nice_orbifolds_centerless}.

\begin{proposition}\label{prop:nice_orbifolds_centerless}
	Let $\Delta$ be a nice $2$-orbifold group. Then $\mathcal{Z}(\Delta) = 1$.
\end{proposition}

\begin{proof}
	First, let's establish a small fact.
	If $G \actson X$, then
	denoting by $\fix_X(g)$ the set of elements fixed by a given $g\in G$ we get
	$$g\cdot \fix_X(h) = \fix_X(ghg^{-1}).$$
	
	Now this fact implies that given $z \in \mathcal{Z}(G)$,
	the set $\fix_X(z)$ is fixed (but not necessarily pointwise!) by the whole group $G$.
	
	Coming back to the nice $2$-orbifold groups,
	consider $\Delta$ as a subgroup of $\isom^+(M)$ where $M = \R^2$ or $\mathbb{H}^2$.
	Then, take $z \in \mathcal{Z}(\Delta)$, assume that $z \ne 1$.
	If $z$ has a fixed point,
	then the whole group $\Delta$ has a global fixed point,
	but this would mean that $\Delta$ is finite cyclic,
	which is a contradiction.
	
	If $z$ has no fixed points,
	regardless of whether $M = \R^2$ or $\mathbb{H}^2$,
	$z$~must have a translation axis.
	However, conjugating $z$ by any elliptic element in $\Delta$ would change it,
	which contradicts the assumption that $z \in \mathcal{Z}$.
\end{proof}

We will use \cref{prop:nice_orbifolds_centerless}
to prove the more difficult \cref{thm:completions_nice_orbifolds_centerless}.

\begin{proposition}\label{thm:completions_nice_orbifolds_centerless}
	Let $\Delta$ be a nice $2$-orbifold group.
	Then $\mathcal{Z}(\widehat \Delta) =1$.
\end{proposition}

\begin{proof}
	First, if $z \in \mathcal{Z}(\widehat \Delta)$ is a torsion element,
	then it is conjugate to a torsion element in $\Delta$,
	which we showed in \cref{prop:nice_orbifolds_centerless} to be centerless.
	Thus $z$ is of infinite order.
	
	By \cref{prop:orbifolds_have_manifolds}
	any $2$-orbifold group $\Delta$ contains a surface group $\Sigma$ as a finite-index subgroup.
	Since $\widehat \Sigma$ is a finite index subgroup of $\widehat \Delta$,
	then for any $g\in \widehat \Delta$
	there is some $k \in \N$ such that $g^k \in \widehat \Sigma$.
	In particular, if $z \in \mathcal{Z}(\widehat\Delta)$ and $z^k \in \widehat\Sigma$,
	we have $z^k \in \mathcal{Z}(\widehat\Sigma)$.
	
	Now, if $\Sigma$ is a fundamental group of a surface of genus $>1$,
	then $\widehat\Sigma$ is centerless as shown in \cite[Proposition 18.]{anderson1974exactness}.
	This would imply that $z^k = 1$, which is a contradiction.
	
	If $\Sigma$ is of genus $0$,
	then $\Delta$ is a finite group,
	which can't be the case as $\Delta$ is nice.
	
	If $\Sigma$ is of genus $1$,
	then $\Delta$ is a $2$-dimensional crystallographic group
	-- a discrete, cocompact subgroup of $\isom^+(\R^2)$.
	By First Bieberbach Theorem \cite[Section 2.1.]{szczepanski2012geometry}
	any $2$-dimensional crystallographic group $\Delta$
	fits in a short exact sequence ${0 \to \Z^2 \to \Delta \to Q \to 1}$,
	where $\Z^2$ is the maximal abelian subgroup consisting of translations
	and $Q$ is finite
	and non-trivial,
	since we assumed that $\Delta \not \cong \Z^2$.
	
	The action of $Q$ on $\Z^2$ must be by finite order elements of $\SL_2(\Z)$
	and the non-trivial finite order elements fix only the zero vector in $\Z^2$.
	Thus, we can choose some $N_0 \in \N$
	such that for all $N \ge N_0$ the action of $Q$ on $(\Z/ N)^2$
	fixes only the zero vector.
	This gives $\mathcal{Z}(\Delta / N\Z^2) = 1$ for large enough $N$.
	
	Since $\widehat \Delta = \varprojlim \Delta / N\Z^2$,
	we get that $\mathcal{Z}(\widehat \Delta) = 1$.
\end{proof}

The only infinite closed orientable 2-orbifold group
which isn't nice, i.e. $\Z^2$,
obviously doesn't satisfy the previous propositions.
However, it turns out that for an $\Z^n$-by-$\Z^2$ extension $G$
unless $G \cong \Z^{n+2}$
the centre is the kernel $\Z^n$ of the extension.
The same holds for the profinite completions
of these extensions.

\begin{proposition}\label{prop:Z2}
	Any central extension of $\Z^n$ by $\Z^2$
	is isomorphic to $H_k \times \Z^{n-1}$
	for some integer $k \ge 0$,
	where $H_k$ has presentation
	\begin{equation*}
		\pres{a, b, c}{[a, c] = [b, c] = 1, [a, b] = c^k}.
	\end{equation*}
	Furthermore, unless $k = 0$, we have
	\begin{equation*}
		\mathcal Z(H_k) = \gen{c}, \qquad
		\mathcal{Z}(\widehat H_k) = \overline{\gen{c}}.
	\end{equation*}
\end{proposition}

\begin{proof}
	The fact that any central extension
	admits such a presentation
	is a consequence of Hopf's Formula
	and the fact that ${H^2(\Z^2; \Z^n) \cong \Z^n}$.
	The number $k$ is such that
	the cohomology class in $\Z^n$
	is a $k$-th multiple of a primitive vector in $\Z^n$.
	
	$H_k$ can also be described as the set $\Z^3$
	with multiplication given by
	\begin{equation*}
		(x_1, y_1, z_1) * (x_2, y_2, z_2) =
		(x_1+x_2, y_1+y_2, z_1+z_2+kx_1y_2).
	\end{equation*}
	Similarily, $\widehat H_k$ is isomorphic
	to the set $\widehat \Z^3$
	with the same multiplication formula.
	In both cases
	\begin{align*}
		[(1, 0, 0), (x, y, z)] & = (0, 0, ky) \\
		[(0, 1, 0), (x, y, z)] & = (0, 0, -kx).
	\end{align*}
	This implies that unless $k=0$,
	the centre is generated by the element $(0, 0, 1)$,
	as both $\Z$ and $\widehat \Z$ are torsion-free.
\end{proof}

\subsubsection{Distinguishing $2$-orbifold groups by profinite completions}

Here we cite a strong result which we will later make use of.
It was stated as \cite[Corollary~4.2.]{wilkes2017profinite},
where it was proved using \cite[Theorem~1.1.]{bridson2013determining}.

\begin{theorem}\label{thm:orbifolds_rigidity}
	Let $O_1, O_2$ be closed $2$-orbifolds. If there is an isomorphism $\widehat{\pi_1(O_1)} \cong \widehat{\pi_1(O_2)}$, then ${\pi_1(O_1) \cong \pi_1(O_2)}$.
\end{theorem}

Notice that our definition of a nice $2$-orbifold requires being closed, so \cref{thm:orbifolds_rigidity} applies to it.

\subsection{Group cohomology}\label{ssec:cohomology}

This section introduces the main results of group cohomology
which we use later.
It is not intended as an introduction to group cohomology
-- for this the reader is referred to Brown's book \cite{brown2012cohomology}
or to the excellent lecture notes of L\"oh \cite{loh2019groupcohomology}
and of Wilkes \cite{wilkes2021profinitegroups}.

\subsubsection{Group extensions}

Group extensions can be a somewhat confusing topic,
due to a lack of uniform terminology.
In this subsection we make sure to properly distinguish
between the isomorphism classes of $N$-by-$Q$ extensions
and the equivalence classes
of short exact sequences $1 \to N \to G \to Q \to 1$
under various equivalence relations.

While it is perfectly possible
to define these for extensions where the kernel isn't central,
or even for extensions where the kernel isn't abelian,
here we will be dealing with \emph{central extensions}
and hence we give all of the definitions in this context.
This is also why we denote the first group in a short exact sequence
$0$ rather than $1$
-- to remind the reader than the kernel is abelian.

\begin{definition}\label{def:equivalent_similar_extensions}
	Let $0 \to M \xrightarrow{\iota_i} G_i \xrightarrow{\pi_i} Q \to 1$
	be two group extensions, with $i=1, 2$.
	We say that they are \emph{equivalent}
	if there is an isomorphism $\tilde\phi: G_1 \to G_2$
	such that the following diagram commutes.
	\begin{equation*}
		\begin{tikzcd}
			0 \arrow{r}
				& M \arrow{r}{\iota_1} \arrow[equal]{d}
					& G_1 \arrow{r}{\pi_1} \arrow{d}{\tilde\phi}
						& Q \arrow{r} \arrow[equal]{d}
							& 1 \\
			0 \arrow{r}
				& M \arrow{r}{\iota_2}
					& G_2 \arrow{r}{\pi_2}
						& Q \arrow{r}
							& 1
		\end{tikzcd}
	\end{equation*}

	If there exists an isomorphism $\tilde \phi: G_1 \to G_2$
	and automorphisms $\Phi: M \to M$ and $\phi: Q \to Q$
	such that the following diagram commutes,
	we say that the extensions are \emph{similar}.
	\begin{equation*}
		\begin{tikzcd}
			0 \arrow{r}
				& M \arrow{r}{\iota_1} \arrow[]{d}{\Phi}
					& G_1 \arrow{r}{\pi_1} \arrow{d}{\tilde\phi}
						& Q \arrow{r} \arrow[]{d}{\phi}
							& 1 \\
			0 \arrow{r}
				& M \arrow{r}{\iota_2}
					& G_2 \arrow{r}{\pi_2}
						& Q \arrow{r}
							& 1
		\end{tikzcd}
	\end{equation*}
\end{definition}

It is clear that equivalent extensions are similar,
but the converse isn't true.
$$\begin{tikzcd}[column sep=2em]
	0 \arrow{r} &
	\Z \arrow{r}{\times 3} &
	\Z \arrow{r} &
	\Z/3 \arrow{r} &
	1
\end{tikzcd}$$
isn't equivalent to
$$\begin{tikzcd}[column sep=2em]
	0 \arrow{r} &
	\Z \arrow{r}{\times (-3)} &
	\Z \arrow{r} &
	\Z/3 \arrow{r} &
	1
\end{tikzcd}$$
but they are similar.

Given an abelian group $M$ and a group $Q$
we will denote by $\mathcal{E}(Q; M)$
the set of equivalence classes of central $M$-by-$Q$ extensions
and by $\overline{\mathcal{E}}(Q; M)$
the set of similarity classes of such extensions.

Now, we finish this section by noting
that for our setting requiring a higher level of detail
from an `$M$-by-$Q$ extension' doesn't produce higher complexity.

\begin{proposition}\label{prop:iso_centerless}
	Let $M$ be an abelian group
	and $Q$ be a centerless group,
	i.e. ${\mathcal{Z}(Q) = 1}$.
	Then, there is a bijection between the set
	$\mathcal E(Q; M)$ of similarity classes
	of central extensions of $M$ by $Q$ and
	the set of isomorphism classes
	of groups $G$ such that $\mathcal Z (G) \cong M$
	and $G/Z(G) \cong Q$.
\end{proposition}

\begin{proof}
	Given any $[0 \to M \to G \to Q \to 1] \in \mathcal{E}(Q; M)$,
	we get an isomorphism class of $G$ in an obvious manner;
	$\mathcal{Z}(Q)=1$ ensures that
	$\mathcal{Z}(G) = \iota(M)$ and $G / \mathcal{Z}(G) \cong Q$.
	Similarly, given any such $G$
	and isomorphisms $f\colon M \to \mathcal{Z}(G)$
	and $g\colon Q \to G / \mathcal{Z}(G)$
	we get
	\begin{equation*}
		\begin{tikzcd}
			0 \arrow{r}	& M \arrow{d}{f} \arrow{r}{\iota \circ f}	& G \arrow[equal]{d} \arrow{r}{g^{-1}\circ \pi}		& Q	\arrow{r} \arrow{d}{g}	& 1 \\
			0 \arrow{r}	& \mathcal{Z}(G) \arrow{r}{\iota}	& G \arrow{r}{\pi}	& G / \mathcal{Z}(G) \arrow{r}	& 	1
		\end{tikzcd}.
	\end{equation*}
	The part which isn't immediate is
	why an isomorphism $\tilde\phi\colon G_1 \to G_2$
	`upgrades' to a similarity of extensions.
	The reason for this is that
	any $\tilde\phi$ restricts to an isomorphism
	$\Phi\colon  \mathcal{Z}(G_1) \to \mathcal{Z}(G_2)$
	and also induces
	$\phi\colon G_1 /\mathcal{Z}(G_1) \to G_2/ \mathcal{Z}(G_2)$
	so that the diagram~(\ref{eq:centerless}) commutes,
	thereby implying that the extensions are similar.
	\begin{equation}\label{eq:centerless}
		\begin{tikzcd}
			0 \arrow{r}	& \mathcal{Z}(G_1) \arrow{r}{\iota_1} \arrow{d}{\Phi}	& G_1 \arrow{r}{\pi_1} \arrow{d}{\tilde\phi}	& G_1 / \mathcal{Z}(G_1) \arrow{r} \arrow{d}{\phi}	& 	1 \\
			0 \arrow{r}	& \mathcal{Z}(G_2) \arrow{r}{\iota_2}	& G_2 \arrow{r}{\pi_2}	& G_2 / \mathcal{Z}(G_2) \arrow{r}	& 	1
		\end{tikzcd}
	\end{equation}
\end{proof}

\subsubsection{$H^2(Q; M)$ classifies extensions} \label{sssec:H2_classifies_extensions}

The most important result for us
is how the second cohomology group classifies
the equivalence classes of group extensions
with given kernel and quotient.
The reader is referred to \cite[Section IV.3.]{brown2012cohomology}.

\begin{proposition}
	For an abelian group $M$ and a group $Q$,
	there is a natural bijection
	between the set $\mathcal{E}(Q; M)$
	of \emph{equivalence} classes of central extensions of $M$ by $Q$
	and the second cohomology group $H^2(Q; M)$ of $Q$
	with coefficients in $M$, treated as a trivial $Q$-module.
\end{proposition}

The naturality of this correspondence is explained in detail
in \cref{prop:from_Loh_discrete}, which we adapt from
\cite[Theorem 1.5.13.]{loh2019groupcohomology}.
We could not find an explicit proof in the topic literature,
which is why we include it here,
and why we prove it allowing for
non-trivial action on the coefficient module.

\begin{proposition}\label{prop:from_Loh_discrete}
	Let $M_1$ be a $Q_1$-module, $M_2$ a $Q_2$-module,
	and $\phi\colon Q_1 \to Q_2$
	and $\Phi\colon M_1 \to M_2$ be homomorphisms
	such that
	$${\Phi(q\cdot m) = \phi(q)\cdot \Phi(m)}.$$
	Given cohomology classes $[\zeta_i] \in H^2(Q_i; M_i)$
	and extension classes
	$$[0 \to M_i \xrightarrow{\iota_i} G_i \xrightarrow{\pi_i} Q_i \to 1]$$
	which they represent,
	the following are equivalent.
	\begin{enumerate}
		\item $\Phi_*([\zeta_1]) = \phi^*([\zeta_2])$
		as elements of $H^2(Q_1; M_2)$.
		\item There is a homomorphism $\tilde\phi\colon G_1 \to G_2$
		such that the diagram~(\ref{eq:extension_isomorphism_prop}) commutes.
		\begin{equation}\label{eq:extension_isomorphism_prop}
			\begin{tikzcd}
				0 \arrow{r}	& M_1 \arrow{r}{\iota_1} \arrow{d}{\Phi} & G_1 \arrow{r}{\pi_1} \arrow{d}{\tilde\phi}	& Q_1 \arrow{r} \arrow{d}{\phi}	& 1 \\
				0 \arrow{r}	& M_2 \arrow{r}{\iota_2}	& G_2 \arrow{r}{\pi_2}	& Q_2 \arrow{r}	& 1
			\end{tikzcd}
		\end{equation}
	\end{enumerate}
\end{proposition}

\begin{proof}[Proof of 1. $\Rightarrow$ 2.]
	The equality $\Phi_*([\zeta_1]) = \phi^*([\zeta_2])$ means
	that there is some $2$-coboundary $d^*_2(g)$ where $g\colon Q \to M$
	such that for all $x, y \in Q_1$ we have
	\begin{equation*}
		\begin{aligned}
		\Phi_*([\zeta_1])(x, y) & =
		\Phi(\zeta_1(x, y)) \\ & = d_2^*(g)(x, y) + \zeta_2(\phi(x), \phi(y))
		= \big(d_2^*(g) + \zeta_2\big)(x, y).
		\end{aligned}
	\end{equation*}
	Now, remembering that
	$G_i$ is the group of elements of the set $M_i \times Q_i$
	with the operation
	$(m_1, q_1)(m_2, q_2)
	= (m_1 + q_1\cdot m_2 + \zeta_i(q_1, q_2), q_1q_2)$,
	let's set
	\begin{equation*}
		\tilde\phi\colon   (m, q) \mapsto (\Phi(m) + g(q), \phi(q)). 
	\end{equation*}
	We then get
	\begin{align*}
		\tilde\phi\big((m_1, q_1)(m_2, q_2)\big)
		& = \tilde\phi\big((m_1 + q_1 \cdot m_2 + \zeta_1(q_1, q_2), q_1q_2)\big)\\
		& = \Big(\Phi\big(m_1 + q_1\cdot m_2 + \zeta_1(q_1, q_2)\big) + g(q_1q_2), \phi(q_1q_2)\Big) \\
		& = \Big(\Phi(m_1)+g(q_1) + \phi(q_1)\cdot(\Phi(m_2) + g(q_2)) \\
		&\phantom{==} + \zeta_2\big(\phi(q_1), \phi(q_2)\big), \phi_1(q_1)\phi_1(q_2)\Big) \\
		& = \big(\Phi(m_1) + g(q_1), \phi(q_1)\big)
		\cdot\big(\Phi(m_2) + g(q_2), \phi(q_2)\big) \\
		& = \tilde\phi\big((m_1, q_1)\big)\cdot \tilde\phi\big((m_2, q_2)\big),
	\end{align*}
	so $\tilde\phi\colon G_1 \to G_2$ is a homomorphism.
	It is clear that the right square of
	diagram~\ref{eq:extension_isomorphism_prop} commutes;
	for the left one notice that
	$$\big(m - \zeta_1(1, 1), 1\big)
	\xmapsto{\tilde\phi}
	\big(\Phi(m) - d_2^*(g)(1,1) - \zeta_2(1, 1) + g(1), 1\big)
	= \iota_2(m),$$
	since $d_2^*(g)(1, 1) = 1\cdot g(1) - g(1) + g(1) = g(1)$.
\end{proof}

\begin{proof}[Proof of 2. $\Rightarrow$ 1.]
	Choose any sections $s_i\colon Q_i \to G_i$
	for $\pi_i\colon G_i \to Q_i$.
	Now, we want to show that
	\begin{equation*}
		h(x, y) := \zeta_2\big(\phi(x), \phi(y)\big) - \Phi\big(\zeta_1(x, y)\big)
	\end{equation*}
	is a $2$-coboundary.
	For this we set $g\colon Q_1 \to G_2$ to be
	$$g(x) := s_2(\phi(x))\tilde\phi(s_1(x))^{-1}.$$
	The image of $g$ gets killed by $\pi_2$,
	so we can treat $g$ as a map from $Q_1$ to $M_2$.
	
	This gives us
	\begin{align*}
		(d_2^*(g))(x, y)
		& = s_2(\phi(x)) \cdot \big(s_2(\phi(y))\tilde\phi(s_1(y))^{-1}\big) \cdot s_2(\phi(x))^{-1} \\
		& \phantom{==} - s_2(\phi(xy))\tilde\phi(s_1(xy))^{-1} + s_2(\phi(x))\tilde\phi(s_1(x))^{-1} \\
		& = \zeta_2(\phi(x), \phi(y)) + s_2(\phi(xy)) \cdot \tilde\phi(s_1(y))^{-1} \cdot s_2(\phi(x))^{-1} \\
		& \phantom{==} + \tilde\phi(s_1(xy))s_2(\phi(xy))^{-1} + s_2(\phi(x))\tilde\phi(s_1(x))^{-1} \\
		& = \zeta_2(\phi(x), \phi(y)) + \tilde\phi(s_1(xy))\cdot \tilde\phi(s_1(y))^{-1} \cdot \tilde\phi(s_1(x))^{-1} \\
		& = \zeta_2(\phi(x), \phi(y)) - \tilde\phi(\zeta_1(x, y)) \\ & = h(x, y),
	\end{align*}
	so $h$ is indeed a $2$-coboundary
	and $\phi^*([\zeta_2]) = \Phi_*([\zeta_1])$.
\end{proof}

This naturality has an important consequence for us.
	
\begin{corollary}\label{prop:actions_commute}
	Let $M$ be an abelian group
	treated as a trivial module for a group $Q$.
	Then the groups $\aut(M)$ and $\aut(Q)$
	act (on the left) on  $H^2(Q; M)$ by
	$$\Phi\cdot [\zeta] = \Phi_*([\zeta]),\quad
	\phi\cdot [\zeta] = (\phi^{-1})^*([\zeta])$$
	and these actions commute.
	Furthermore, the orbits of
	the action of $\aut(M)\times \aut(Q)$ on $H^2(Q; M)$
	are in bijection with the set $\overline{\mathcal{E}}(Q; M)$
	of similarity classes of central extensions of $M$ by $Q$.
\end{corollary}

\subsubsection{Cohomological Hopf's Formula}\label{sssec:Hopfs_formula}

In \cref{sssec:H2_classifies_extensions}
we discussed the naturality of how $H^2(Q; M)$ classifies the equivalence classes of group extensions,
but we will later need an explicit way of computing
the action of $\aut(Q)$ on $H^2(Q; M)$ done with a specific free resolution
-- one coming from the presentation complex.
In order to do that we introduce a version of Cohomological Hopf's Formula
applied to the partial resolution
coming from the presentation complex of $Q$.

The main tools for this are
the Inflation-Restriction Exact Sequence
and the natural correspondence
between $H^1(Q; M)$ and the group of homomorphisms $Q \to M$.

The following statement of the Inflation-Restriction Exact Sequence
was taken from \cite[Proposition 1.6.6.]{neukirch2013cohomology}
and it holds in the context of
modules $M$ with possibly non-trivial action of $G$.
We use superscript $M^G$ to denote the $G$-invariants submodule.

\begin{theorem}[Inflation-Restriction Exact Sequence]\label{thm:inflation-restriction}
	Given a short exact sequence of groups
	$1 \to N \to G \to Q \to 1$
	and a $G$-module $M$,
	there is an exact sequence
	\begin{equation*}
		\begin{tikzcd}[column sep=2em]
			0 \arrow{r}
			& H^1(Q; M^N) \arrow{r}
			& H^1(G; M) \arrow{r}
			& H^1(N; M)^Q \arrow[in=180, out=0, looseness=1.5, overlay]{lld}
			\\
			& H^2(Q; M^N) \arrow{r}
			& H^2(G; M)
			&
		\end{tikzcd}
	\end{equation*}
	where the maps $H^k(Q; M^N) \to H^k(G; M)$ are the \emph{inflation} maps,
	the maps $H^k(G; M) \to H^k(N; M)^Q$ are the \emph{restriction} maps
	and the map $\Tr: H^1(N; M)^Q \to H^2(Q; M^N)$
	is the \emph{transgression} map given by the following procedure.
	
	Given a $1$-cocycle $g: N \to M$
	whose cohomology class is $Q$-invariant,
	there is a $1$-cochain $f: G \to M$ such that
	\begin{enumerate}
		\item $f|_N = g$,
		\item $(d^*f): G \times G \to M$ is constant on the cosets of $N \times N$,
		\item $\im (d^*f)$ lies in $M^N$.
	\end{enumerate}
	Since $(d^*f)$ is constant on the cosets of $N\times N$,
	we can consider it as a function $Q\times Q \to M^N$
	and set
	$$\Tr([g]) := [(d^*f): Q \times Q \to M^N] \in H^2(Q; M^N).$$
\end{theorem}

\Cref{prop:H1_classifies_homs} is a well-known result
and it is easy to see when one computes $H^1(Q; M)$ with the standard (bar) resolution:
the $1$-cocycles \emph{are} the homomorphisms $Q \to M$,
while the $1$-coboundaries are all zero maps.

\begin{lemma}\label{prop:H1_classifies_homs}
	Let $M$ be an abelian group treated as a trivial $G$-module.
	Then there is a natural bijection
	$$H^1(Q; M) \rightarrow \{\text{homomorphisms}\ f: Q \to M\}.$$
\end{lemma}

Finally, we can state the explicit version of Cohomological Hopf's Formula.
While the existence of the isomorphism in it is well-known,
the explicit isomorphism that we need is hard to find in the literature.

\begin{proposition}[Cohomological Hopf's Formula]\label{prop:Hopfs_formula}
	Let $Q$ be a group with presentation
	$$\pres{a_i: i\in I}{r_j: j\in J}$$
	and $M$ be an abelian group treated as a trivial $Q$-module.
	Let $F$ be the free group on the set $\{a_i: i\in I\}$
	and $R$ be the smallest normal subgroup of $F$ containing the set $\{r_j: j \in J\}$.
	
	Let $C_3 \to C_2 \to C_1 \to C_0 \to \Z$
	be any partial free resolution of the presentation resolution
	with $C_2 = \bigoplus_{j\in J} \Z Q \cdot \textbf{\textsc{r}}_j$
	and $C_1 = \bigoplus_{i\in I} \Z Q \cdot \textbf{\textsc{a}}_i$.
	
	Then the following map $\Psi$ is a natural bijection, where
	\begin{equation*}
		\begin{aligned}
			\frac{\big\{F-\text{invariant homs.}\ R \to M\big\}}%
			{\big\{\text{restrictions of homs.}\ F\to M\big\}}
			\xrightarrow{\Psi} \\
			\frac{\ker \big(\hom_{\Z Q}(C_2, M) \to \hom_{\Z Q}(C_3, M)\big)}%
			{\im \big(\hom_{\Z Q}(C_1, M) \to \hom_{\Z Q}(C_2, M)\big)}
		\end{aligned}
	\end{equation*}
	is given by
	$$\Psi([g: R \to M]) = \big[\textbf{\textsc{r}}_j \mapsto g(r_j): j \in J\big].$$
\end{proposition}

\begin{proof}
	We use the Inflation-Restriction Exact Sequence from \cref{thm:inflation-restriction}
	for the short exact sequence $1 \to R \to F \xrightarrow{\pi} Q \to 1$
	and $M$ treated as an $F$-module.
	
	Since $F$ is of cohomological dimension $1$, we get $H^2(F; M) = 0$ and so
	the transgression map $\Tr: H^1(R; M)^Q \to H^2(Q; M)$ is surjective
	with $\ker \Tr$ being equal to the image of $H^1(F; M)$ under the restriction map.
	Now, $H^1(R; M)^Q$ is the set of $F$-invariant homomorphisms $R \to M$,
	while $H^1(F; M)$ is the set of all homomorphisms $F \to M$.
	This means that sending $[g: R \to M]$ to $\Tr([g]) \in H^2(Q; M)$
	gives a natural isomorphism
	$$\frac{\big\{F-\text{invariant homs.}\ R \to M\big\}}%
	{\big\{\text{restrictions of homs.}\ F\to M\big\}} \to H^2(Q; M).$$
	
	We now need to compose this natural isomorphism
	with the natural isomorphism
	$$\Theta: H^2(Q; M) \to
	\frac{\ker \big(\hom_{\Z Q}(C_2, M) \to \hom_{\Z Q}(C_3, M)\big)}%
	{\im \big(\hom_{\Z Q}(C_1, M) \to \hom_{\Z Q}(C_2, M)\big)}$$
	coming from computing the cohomology with two different
	(partial) free resolutions of $\Z$ by $\Z Q$-modules:
	the bar resolution and the presentation resolution.
	
	\begin{equation*}
		\begin{tikzcd}
			\bigoplus_{j\in J} \Z Q \cdot \textbf{\textsc{r}}_i \arrow{r} \arrow{d}{\Theta_2}
				& \bigoplus_{i\in I} \Z Q \cdot \textbf{\textsc{a}}_i \arrow{r} \arrow{d}{\Theta_1}
					& \Z Q \arrow{r} \arrow[equal]{d}
						& \Z \arrow[equal]{d} \\
			\bigoplus_{q_1, q_2 \in Q} \Z Q \cdot [q_1|q_2] \arrow{r}
			& \bigoplus_{q\in Q} \Z Q \cdot [q] \arrow{r}
			& \Z Q \cdot [] \arrow{r}
			& \Z 
		\end{tikzcd}
	\end{equation*}
	
	Defining $\Theta_1$ is straightforward:
	we set $\textbf{\textsc{a}}_i \mapsto [a_i]$.
	$\Theta_2$ poses a bigger challenge.
	First let's notice that for a relation $r_j = a_{i_1}a_{i_2}\ldots a_{i_l}$
	we have
	$$d(\textbf{\textsc{r}}_j)
	= \textbf{\textsc{a}}_{i_1} + a_{i_1} \cdot \textbf{\textsc{a}}_{i_2} + \ldots
	+ a_{i_1}a_{i_2}\ldots a_{i_{l-1}} \cdot \textbf{\textsc{a}}_{i_l}.$$
	Let's define $\Theta_2$ by `triangulating the relation disc' as follows.
	$$\Theta_2(\textbf{\textsc{r}}_j)
	= [1|a_{i_1}] + [a_{i_1}| a_{i_2}] + [a_{i_1}a_{i_2}|a_{i_3}] + \ldots
	+ [a_{i_1}a_{i_2}\ldots a_{i_{l-1}} | a_{i_l}]$$
	Then we get
	\begin{align*}
		(d\circ \Theta_2)(\textbf{\textsc{r}}_j)
		& = \left(1 \cdot [a_{i_1}] - [a_{i_1}] + [1]\right)
		+ \left(a_{i_1}\cdot[a_{i_2}] - [a_{i_1}a_{i_2}] + [a_{i_1}]\right) \\
		& + \left(a_{i_1}a_{i_2}\cdot [a_{i_3}] - [a_{i_1}a_{i_2}a_{i_3}] + [a_{i_1}a_{i_2}]\right)
		+ \ldots \\
		& + \left(a_{i_1}a_{i_2}\ldots a_{i_{l-1}}\cdot [a_{i_l}] - [1] + [a_{i_1}a_{i_2}\ldots a_{i_{l-1}}]\right) \\
		& = 1 \cdot [a_{i_1}] + a_{i_1} \cdot [a_{i_2}] + a_{i_1}a_{i_2} \cdot [a_{i_3}] + \ldots \\
		& + a_{i_1}a_{i_2}\ldots a_{i_{l-1}} \cdot [a_{i_l}]
		= (\Theta_1 \circ d)(\textbf{\textsc{r}}_j),
	\end{align*}
	which means that $\Theta_{*}$ is indeed a chain map
	and since it extends an isomorphism $\Z Q \to \Z Q \cdot []$,
	it extends to a chain equivalence.
	
	Now, we need to compute the map $\Theta$.
	Given $h \in C^2(Q; M)$ we get
	$$\Theta\left(h\right)\left(\textbf{\textsc{r}}_j\right)
	= h(1, a_{i_1}) + h(a_{i_1}, a_{i_2}) + \ldots
	+ h(a_{i_1}a_{i_2}\ldots a_{i_{l-1}}, a_{i_l}).$$
	
	Finally, let's compute the composition $\Theta\circ\text{Tr}$.
	Given an $F$-invariant homomorphism $g: R \to M$
	and choosing a $1$-cochain $f: F \to M$
	such that $\Tr([g]) = [(d^*f)]$
	we get the following.
	
	\begin{align*}
		\Theta
		\left(\text{Tr}\left(g\right)\right)\left(\textbf{\textsc{r}}_j\right)
		& = \text{Tr}\left(g\right)(1, a_{i_1}) + \ldots
		+ \text{Tr}\left(g\right)(a_{i_1}a_{i_2}\ldots a_{i_{l-1}}, a_{i_l}) \\
		& = \big(1 \cdot f( a_{i_1}) - f(1\cdot  a_{i_1}) + f(1)\big) \\
		& + \big( a_{i_1} \cdot f( a_{i_2}) - f( a_{i_1} \cdot  a_{i_2}) + f( a_{i_1})\big)  + \ldots \\
		& + \big( a_{i_1} \ldots  a_{i_{l-1}} \cdot f( a_{i_l})
		- f( a_{i_1} \ldots  a_{i_l})
		+ f( a_{i_1} \ldots  a_{i_{l-1}})\big)\\
		& = f(1) - f(r_j) \\
		& + 1 \cdot f( a_{i_1}) +  a_{i_1} \cdot f( a_{i_1}  a_{i_2}) + \ldots
		+ a_{i_1} \ldots  a_{i_{l-1}} \cdot f( a_{i_l})\\
		& = - g(r_j) + (d^*f)(\textbf{\textsc{r}}_j)
	\end{align*}
	Thus we get that $\Psi\big([g]\big) = - \Theta\circ \text{Tr}$.
	Furthermore, since $\Theta$ and $\text{Tr}$ are natural isomorphisms,
	then so is $\Psi$.
\end{proof}

\subsubsection{Cohomology of profinite groups}\label{sssec:profinite_cohomology}

For profinite groups the usual group cohomology doesn't capture the topological data.
However, if we require the appropriate maps to be continuous, we do get valuable information.

It turns out that for reasons coming from homological algebra,
it isn't obvious which topological modules one can allow
to make the `profinite homology and cohomology' well-behaved theories.
To avoid complications,
we follow the definitions of \cite[Section 6.2.]{ribes2000profinite}
specified to finite coefficient modules,
but first we need to define the completed group algebra $\widehat \Z[[\Gamma]]$
for a profinite group $\Gamma$.

\begin{definition}
	Given a profinite group $\Gamma$, we define the \emph{completed group algebra}
	$\widehat{\Z[\Gamma]} = \widehat \Z[[\Gamma]]$ to be the limit
	\begin{equation*}
		\widehat \Z[[\Gamma]] := \varprojlim_{N \in \N,\ U \triangleleft_{o} \Gamma} (\Z / N)[\Gamma / U].
	\end{equation*}
\end{definition}

Now, we need to define two types of topological $\widehat\Z[[\Gamma]]$-modules.

\begin{definition}
	For a profinite group $\Gamma$ a topological $\widehat\Z[[\Gamma]]$-module $M$ is called \emph{discrete} if it has discrete topology and \emph{profinite} if it has profinite topology.
\end{definition}

This allows us to give a definition for the cohomology of profinite groups.
Being of homological algebra nature it is somewhat abstract,
but it ensures all functorial properties that we require from cohomology.

\begin{definition}\label{def:profinite_cohomology}
	For a profinite group $\Gamma$ and a \emph{discrete} $\widehat \Z[[\Gamma]]$-module $M$ we set
	$$\widehat H^k(\Gamma; M) := \Ext_{\widehat \Z[[\Gamma]]}^k(\widehat \Z; M).$$
\end{definition}

In practice, we can compute $\widehat H^k(\Gamma; M)$ by taking a projective resolution of $\widehat \Z$ by profinite $\widehat\Z[[\Gamma]]$-modules of length at least $k+1$, applying the functor $\hom_{\widehat\Z[[\Gamma]]}(-, M)$ and computing the homology of the resulting chain complex.

In principle one can take the analogue of the bar resolution
for the projective resolution required in \cref{def:profinite_cohomology},
but this wouldn't be convenient for calculations in practice.
However, for \emph{good} groups
it turns out to be feasible
to convert a `small' free resolution for $G$
into a `small' free resolution for $\widehat G$.
We describe it in \cref{lem:hat_resolution_from_discrete},
which we take from \cite[Proposition 3.14.]{wilkes2017profinite},
but first we define goodness.

\begin{definition}\label{def:goodness}
	A group $G$ is called \emph{good} if the canonical map
	$$\iota^*\colon  \widehat H^2(\widehat G; M) \to H^2(G; M)$$
	is an isomorphism for any finite $G$-module $M$.
\end{definition}

\begin{lemma}\label{lem:hat_resolution_from_discrete}
	Let $G$ be a good group
	with a partial resolution $(C_i)_{0 \le i \le n}$ of $\Z$
	by finitely generated free $\Z G$-modules
	$C_i = \bigoplus_{j=1}^{l_i}\Z G \cdot \mathbf{x}_{ij}$.
	Then $(\widehat C_i)_{0 \le i \le n}$
	is a partial resolution of $\widehat \Z$ by free $\widehat\Z[[\widehat G]]$-modules,
	where
	$$\widehat C_i = \bigoplus_{j=1}^{l_i}\widehat\Z[[\widehat G]]\cdot \mathbf{x}_{ij}.$$
\end{lemma}

Finally, we need to note that
among the mentioned `functorial properties we require from cohomology'
are the following two.

Firstly, Inflation-Restriction Exact Sequence still holds
for (continuous) exact sequences of profinite groups.

Secondly, given a profinite group $\Gamma$ and a \emph{finite} group $M$,
the group $\widehat H^2(\Gamma; M)$ is in natural bijection
with the equivalence classes of \emph{profinite} extensions of $M$ by $\Gamma$
-- these are the short exact sequences
$$0 \to M \xrightarrow{\iota} E \xrightarrow{\pi} \Gamma \to 1$$
where $E$ is a profinite group
and the maps $\iota$ and $\pi$ are continuous homomorphisms.
Analogously, we define \emph{equivalence} and \emph{similarity} of profinite extensions
by requiring that the isomorphisms in \cref{def:equivalent_similar_extensions} be continuous.

\subsubsection{Cohomology of nice $2$-orbifold groups}\label{sssec:cohomology_computed}

In this section we compute the cohomology groups
for nice $2$-orbifold groups,
taking coefficient modules with trivial action.

We base our computation
on a partial free resolution of length $3$
given in \cite[Proposition 6.6.]{wilkes2017profinite}.

\begin{proposition}\label{prop:presentation_complex_extended}
	The cellular chain complex
	for the Cayley complex $C_\mathcal{P} = \widetilde{K_{\mathcal{P}}}$
	of the presentation~(\ref{eq:presentation}) for a nice $2$-orbifold group
	\begin{equation*}
		\begin{tikzcd}[column sep=1em]
			\underbrace{\Z \Delta\{\textsc{\textbf{r}}_0, \ldots, \textsc{\textbf{r}}_m\}}_{C_2} \arrow{r}{d}
			& \underbrace{\Z \Delta\{\textsc{\textbf{x}}_1, \ldots, \textsc{\textbf{y}}_g, \textsc{\textbf{a}}_1, \ldots, \textsc{\textbf{a}}_m\}}_{C_1} \arrow{r}{d}
			& \underbrace{\Z \Delta}_{C_0} \arrow{r}{d}
			& \Z
		\end{tikzcd}
	\end{equation*}
	can be resolved one dimension further
	by setting
	$C_3 = \Z \Delta \{\textsc{\textbf{z}}_1, \ldots, \textsc{\textbf{z}}_m\}$
	where the map $d\colon C_3 \to C_2$ is defined as
	$$d(\textsc{\textbf{z}}_i) = (a_i - 1)\cdot \textsc{\textbf{r}}_i.$$
\end{proposition}

Because the extensions we study are central,
it will prove useful to look at the chain complex
in \cref{prop:presentation_complex_extended}
after applying the functor $\Z \otimes_{\Z \Delta}(-)$.

\begin{proposition}\label{prop:presentation_complex_after_trivial_action}
	Applying the functor
	$$\Z \otimes_{\Z \Delta} (-)\colon
	\Z \Delta\text{-chain complexes} \to \Z\text{-chain complexes}$$
	to the complex $C_3 \to C_2 \to C_1 \to C_0$
	from \cref{prop:presentation_complex_extended}
	gives commutative diagram~\ref{eq:after_trivial_action}.
	The maps $D_3 \to D_2$ and $D_1 \to D_0$ are zero,
	and the map ${d\colon D_2 \to D_1}$ is given by
	$$\mathbf{r}_0 \mapsto \sum_{i=1}^m \mathbf{a}_i\quad
	\text{and}\quad \mathbf{r}_i \mapsto p_i \cdot \mathbf{a}_i.$$
	
	\begin{equation}\label{eq:after_trivial_action}
		\begin{tikzcd}[column sep=1em]
			\underbrace{\Z \Delta\{\textsc{\textbf{z}}_1, \ldots, \textsc{\textbf{z}}_m\}}_{C_3} 
				\arrow{r}{d}
				\arrow[swap]{d}{1\otimes (-)}
			& \underbrace{\Z \Delta\{\textsc{\textbf{r}}_0, \ldots, \textsc{\textbf{r}}_m\}}_{C_2} 
				\arrow{r}{d}
				\arrow[swap]{d}{1\otimes (-)}
			& \underbrace{\Z \Delta\{\textsc{\textbf{x}}_1, \ldots, \textsc{\textbf{a}}_m\}}_{C_1} 
				\arrow{r}{d}
				\arrow[swap]{d}{1\otimes (-)}
			& \underbrace{\Z \Delta}_{C_0}
				\arrow[swap]{d}{1\otimes (-)} \\
			\underbrace{\Z\{\mathbf{z}_1, \ldots, \mathbf{z}_m\}}_{D_3} \arrow{r}{0} & \underbrace{\Z\{\mathbf{r}_0, \ldots, \mathbf{r}_m\}}_{D_2}\arrow{r}{d} & \underbrace{\Z\{\mathbf{x}_1, \ldots, \mathbf{a}_m\}}_{D_1}\arrow{r}{0} & \underbrace{\Z}_{D_0}
		\end{tikzcd}
	\end{equation}
\end{proposition}

\begin{proof}
	The boundary maps $D_3 \to D_2$ and $D_1 \to D_0$ in the bottom row
	became zero maps,
	because $d(\textbf{\textsc{z}}_i) = (a_i -1)\cdot \textbf{\textsc{r}}_i$
	and $1 \otimes (a_i - 1) \cdot \textbf{r}_i = 0\cdot \textbf{r}_i$,
	and similarly for the map in degree $1$.
\end{proof}

It is worth noting
that the map
$d\colon  \Z\{\textbf{r}_0, \ldots, \textbf{r}_m\}
\to \Z\{\textbf{x}_1, \ldots, \textbf{a}_m\}$
is described by the matrix~(\ref{eq:matrix_chain}).
\begin{equation}\label{eq:matrix_chain}
	\begin{pmatrix}
		0		&	0		&	0		&	\dots	&	0		\\
		\vdots	&	\vdots	&	\vdots	&	\ddots	&	\vdots	\\
		0		&	0		&	0		&	\dots	&	0		\\
		1		&	p_1		&	0		&	\dots	&	0		\\
		1		&	0 		&	p_2		&	\dots	&	0		\\
		\vdots	&	\vdots	&	\vdots	&	\ddots	&	\vdots	\\
		1		&	0		&	0		&	\dots	&	p_m		\\
	\end{pmatrix}
\end{equation}

\begin{proposition}\label{prop:H2_of_nice_orbifolds}
	Given a nice $2$-orbifold group $\Delta$
	with presentation~(\ref{eq:presentation}),
	the second cohomology group $H^2(\Delta; M)$ of $\Delta$
	with trivial coefficient module $M$
	is 
	\begin{equation}\label{eq:H2_explicit}
		H^2(\Delta; M)
		\cong \bigoplus_{i=1}^{m}M \cdot \mathbf{r}_i^*
		/ (\mathbf{r}_0^* + p_i \cdot \mathbf{r}_i^*\ \text{for}\ i = 1, \ldots, m).
	\end{equation}
\end{proposition}

\begin{proof}
	Applying functor $\hom_{\Z\Delta}(-, M)$
	to the partial resolution
	$${C_3 \to C_2 \to C_1 \to C_0}$$
	from~\cref{prop:presentation_complex_extended}
	factors through the functor $\Z \otimes_{\Z \Delta} (-)$,
	as shown in \cref{prop:presentation_complex_after_trivial_action},
	so we only need to apply the functor $\hom_{\Z}(-, M)$
	to the chain complex ${D_3 \to D_2 \to D_1 \to D_0}$
	of \cref{prop:presentation_complex_after_trivial_action}.
	This gives us the following chain complex.
	\begin{equation*}
		\begin{tikzcd}[column sep=1em]
			\underbrace{M\{\textbf{z}_1^*, \ldots, \textbf{z}_m^*\}}_{\hom(D_3, M)} & \underbrace{M\{\textbf{r}_0^*, \ldots, \textbf{r}_m^*\}}_{\hom(D_2, M)}\arrow{l}{0} & \underbrace{M\{\textbf{x}_1^*, \ldots, \textbf{a}_m^*\}}_{\hom(D_1, M)}\arrow{l}{d^*} & \underbrace{M}_{\hom(D_0, M)} \arrow{l}{0}
		\end{tikzcd}
	\end{equation*}
	Now,
	the kernel of the map $d^*: \hom(D_2, M) \to \hom(D_3, M)$
	is the whole domain $\hom(D_2, M)$
	and so $H^2(\Delta; M)$ is the cokernel of $d^*$.
	Since all of $\mathbf{x}_i^*$ and $\mathbf{y}_i^*$ are mapped to zero,
	this cokernel is precisely
	$$\bigoplus_{i=1}^{m}M \cdot \mathbf{r}_i^* 
	/ (\mathbf{r}_0^* + p_i \cdot \mathbf{r}_i^*\ \text{for}\ i = 1, \ldots, m).$$
\end{proof}

Now we can compute the action
$\aut(M)\times\aut(\Delta)\actson H^2_\mathcal{P}(\Delta; M)$,
as introduced in \cref{prop:actions_commute}.

\begin{proposition}\label{prop:action_on_discrete_H2}
	Let $\Phi \in \aut(M)$ and $\phi\in \aut(\Delta)$ be automorphisms,
	and ${[\zeta] = [\sum_{i=0}^{m} m_i \cdot \mathbf{r}_i^*]}$
	be an element of $H^2(\Delta; M)$.
	Then
	$$\Phi \cdot [\zeta] = [\sum_{i=0}^{m} \Phi (m_i) \cdot \mathbf{r}_i^*]$$
	and
	$$\phi^{-1} \cdot [\zeta] = \phi^*([\zeta])
	= \kappa \cdot [m_0 \cdot \mathbf{r}_0^* + \sum m_i \cdot \mathbf{r}_{\sigma(i)}],$$
	where $\phi(a_i) =\ ^{g_i}\big(a_{\sigma(i)}^{k_i}\big)$
	and $\kappa \in \Z$ is such that $\kappa \equiv k_i \mod p_i$.
	
	Furthermore if $M \cong \Z^r\oplus\ T$ with $r\ge 1$ and $T$ being torsion,
	then $\kappa \in \Z^\times = \{\pm 1\}$.
\end{proposition}

The proof of this theorem follows the proof in~\cite{wilkes2017profinite},
with the exception of changing $\Z$ to $M$.

\begin{proof}
	The action of $\Phi$ is
	the usual action of the automorphism group of the coefficient module.
	
	For computing $\phi^*$,
	let's denote by $F$ the free group
	generated by the letters $x_1, \ldots, y_g, a_1, \ldots, a_m$
	and by $R$ the kernel of $v\colon F \to \Delta$,
	so that the sequence
	$$1 \to R \to F \xrightarrow{v} \Delta \to 1$$
	is exact.
	To compute the effect that $\phi$ induces on $H^2(\Delta; M)$,
	we want to use the naturality of Cohomological Hopf's Formula
	as in \cref{prop:Hopfs_formula}
	and compute it on the set
	$$\frac{\{F-\text{invariant homomorphisms}\ f\colon R \to M\}}%
	{\{\text{restrictions of homomorphisms}\ g\colon F \to M\}}.$$
	
	We need to choose a map $\tilde{\phi}\colon  F \to F$
	such that the following diagram commutes.
	\begin{equation*}
		\begin{tikzcd}
			F \arrow{r}{\tilde{\phi}}\arrow{d}{v} & F \arrow{d}{v} \\
			\Delta \arrow{r}{\phi} & \Delta
		\end{tikzcd}
	\end{equation*}
	For easy calculation we choose
	$\tilde\phi({a}_i) =\ ^{g_i}({a}_{\sigma(i)}^{k_i})$
	on the generators $a_i$
	and extend it to the other generators
	in any way which makes the diagram commute.
	This is a valid choice as
	$v\big(^{g_i}({a}_{\sigma(i)}^{k_i})\big) =\
	^{g_i}({a}_{\sigma(i)}^{k_i})$
	as elements of $\Delta$.
	
	Now, for $i = 1, \ldots, m$ we can easily compute the effect on the $r_i$.
	\begin{equation*}
		r_i = a_i^{p_i} \xmapsto{\tilde\phi} (\ ^{g_i}a_{\sigma(i)}^{k_i})^{p_i} =\ ^{g_i}(a_{\sigma(i)}^{p_i})^{k_i} =\ ^{g_i}\big(r_{\sigma(i)}^{k_i}\big).
	\end{equation*}
	
	This, perhaps surprisingly,
	gives us enough information to compute the action on $H^2(\Delta; M)$.
	On the level of $\Z \Delta$ chain complex $C_*$ we have
	\begin{align*}
		\phi_*(\textsc{\textbf{a}}_i) & = k_ig_i \cdot \textsc{\textbf{a}}_{\sigma(i)} \\
		\phi_*(\textsc{\textbf{r}}_i) & = k_ig_i \cdot \textsc{\textbf{r}}_{\sigma(i)}.
	\end{align*}
	Since $g_i$ is a group element,
	it acts the same way as $1$ on $D_*$,
	so after applying functor $\Z \otimes_{\Z \Delta}(-)$
	we get
	\begin{align*}
		\phi_*(\textbf{a}_i) & = k_i \cdot \textbf{a}_{\sigma(i)} \\
		\phi_*(\textbf{r}_i) & = p_i \cdot \textbf{r}_{\sigma(i)} \text{ for } i = 1, 2, 3 \\
		\phi_*(\textbf{r}_0) & = \kappa\cdot \textbf{r}_0 +
		\sum_{i = 1}^m \mu_i \cdot \textbf{r}_{\sigma(i)}
		\text{ for some } \kappa, \mu_i \in \Z.
	\end{align*}
	Since $\phi_*$ and the boundary maps $d$ commute,
	we can use it to understand something about the $\kappa$ and $\mu_i$.
	\begin{align*}
		(\phi_*\circ d)\ (\textbf{r}_0)
		& = \phi_*(\textbf{a}_1 + \ldots + \textbf{a}_m) 
		= \sum_{i = 1}^m k_i \cdot \textbf{a}_{\sigma(i)} \\
		(d \circ \phi_*)\ (\textbf{r}_0)
		& = d(\kappa\cdot \textbf{r}_0 + \sum_{i = 1}^m \mu_i \cdot \textbf{r}_{\sigma(i)}) 
		= \sum_{i = 1}^m (\kappa + p_i\mu_i) \cdot \textbf{a}_{\sigma(i)}
	\end{align*}
	so we get $k_i = \kappa + p_i\mu_i$ for $i=1, \ldots, m$.
	
	To see what happens at the cohomology level,
	let's compute what happens after applying $\hom(-, M)$.
	
	We will write $m\cdot \mathbf{r}_i^*$
	for the unique map $C_2 \to M$
	which sends $\mathbf{r}_i$ to $m$
	and all other $\mathbf{r}_j$ to $0$.
	\begin{align*}
		\phi^*(m \cdot \textbf{r}_0^*) &= \kappa\cdot m \cdot \textbf{r}_0^* \\
		\phi^*(m \cdot \textbf{r}_i^*) &= \mu_i \cdot m \cdot \textbf{r}_0^* + k_i \cdot m \cdot \textbf{r}_{\sigma(i)}^* \\
		d^*(m \cdot \textbf{a}_i^*) & = m \cdot \textbf{r}_0^* + p_i \cdot m \cdot \textbf{r}_i^*
	\end{align*}
	
	Finally, let's compute the action of $\phi^*$ on cohomology.
	\begin{align*}
		\phi^*([m \cdot \textbf{r}_0^*]) 	& = \kappa \cdot [m \cdot \textbf{r}_0^*] \\
		\phi^*([m \cdot \textbf{r}_i^*]) 	& = \mu_i \cdot [m \cdot \textbf{r}_0^*] + k_i \cdot [m \cdot \textbf{r}_{\sigma(i)}^*] \\
		& = \mu_i \cdot [m \cdot \textbf{r}_0^*] + (\kappa + p_i\mu_i) \cdot [m \cdot \textbf{r}_{\sigma(i)}^*] \\
		& = \mu_i \cdot [m \cdot \textbf{r}_0^*] + \kappa \cdot [m \cdot \textbf{r}_{\sigma(i)}^*] + \mu_i\cdot [p_i\cdot m \cdot \textbf{r}_{\sigma(i)}^*] \\
		& = \kappa \cdot [m \cdot \textbf{r}_{\sigma(i)}^*] + \mu_i\cdot [m\cdot \textbf{r}_0^* + p_i \cdot m \cdot \textbf{r}_{\sigma(i)}^*] \\
		& = \kappa \cdot [m \cdot \textbf{r}_{\sigma(i)}^*] + \mu_i\cdot [m\cdot \textbf{r}_0^* + p_{\sigma(i)} \cdot m \cdot \textbf{r}_{\sigma(i)}^*] \\
		& = \kappa \cdot [m \cdot \textbf{r}_{\sigma(i)}^*] + \mu_i \cdot [d^*(m\cdot \textbf{a}_{\sigma(i)})] = \kappa \cdot [m \cdot \textbf{r}_{\sigma(i)}^*]
	\end{align*}
	
	Thus, the action is
	by multiplication by $\kappa \in \Z$
	together with permutation of the generators $\textbf{r}_i^*$
	for the cocycles $Z^2(B; M)$.
	
	Taking the $(m!)$-th power of $\phi$
	to get rid of the permutation $\sigma$,
	we get
	$$\big(\phi^{m!}\big)^* [\zeta] = \kappa^{m!} \cdot[\zeta]$$
	and if $M \cong \Z^r \oplus\ T$,
	then $H^2(\Delta; M)$ contains a copy of $\Z^r$ too
	and so ${\kappa = \pm 1}$
	since $(\phi^{m!})^*$ is an automorphism.
\end{proof}

Now, we can do the very similar calculations
for the profinite cohomology defined in \cref{sssec:profinite_cohomology}.

\begin{proposition}\label{prop:H2_hat_of_nice_2-orbifolds}
	Given a nice $2$-orbifold group $\Delta$
	with presentation~(\ref{eq:presentation})
	and a finite trivial $\widehat \Delta$-module $M$,
	the second profinite cohomology group is
	\begin{equation}\label{eq:H2_profinite_explicit}
		\widehat H^2(\widehat \Delta; M)
		= \bigoplus_{i=1}^{m}M \cdot \mathbf{r}_i^*\
		/ (\mathbf{r}_0^* + p_i \cdot \mathbf{r}_i^*\ \text{for}\ i = 1, \ldots, m).
	\end{equation}
\end{proposition}

\begin{proof}
	The partial free resolution
	from \cref{prop:presentation_complex_extended}
	gives rise to a partial free resolution~\ref{eq:complex_after_hats} of $\widehat \Z$
	by $\widehat \Z[[\widehat\Delta]]$-modules
	by \cref{lem:hat_resolution_from_discrete}.
	We can use it to compute the $\widehat H^2(\widehat \Delta; \widehat \Z^n)$.
	\begin{equation}\label{eq:complex_after_hats}
		\begin{tikzcd}
			\widehat C_3 \arrow{r}{\widehat d_3}	& \widehat C_2 \arrow{r}{\widehat d_2}	& \widehat C_1 \arrow{r}{\widehat d_1}	& \widehat C_0 \arrow{r}{\widehat d_0}	& \widehat\Z
		\end{tikzcd}
	\end{equation}
	
	Applying the functor $\hom_{\widehat\Z[[\widehat \Delta]]}(-, M)$
	to the resolution~\ref{eq:complex_after_hats}
	factors through the functor $\widehat \Z \otimes_{\widehat \Z[[\widehat \Delta]]}(-)$,
	similarly to \cref{prop:presentation_complex_after_trivial_action}.
	
	Thus, we only need to apply the functor $\hom_{\widehat \Z}(-, M)$
	to the chain complex $\widehat D_3 \to \widehat D_2 \to \widehat D_1 \to \widehat D_0$
	of $\widehat D_i = \widehat \Z \otimes_{\widehat \Z[[\widehat \Delta]]}\widehat C_i$,
	where incidently $\widehat D_i$ is the profinite completion of the abelian group $D_i$
	from \cref{prop:presentation_complex_after_trivial_action}
	treated as a trivial $\widehat\Z[[\widehat \Delta]]$-module.
	
	In particular, we get the commutative diagram~\ref{eq:profinite_after_trivial_action},
	where we write $R$ for $\widehat \Z[[\widehat \Delta]]$.
	\begin{equation}\label{eq:profinite_after_trivial_action}
		\begin{tikzcd}[column sep=1em]
			\underbrace{R\{\textsc{\textbf{z}}_1, \ldots, \textsc{\textbf{z}}_m\}}_{\widehat C_3}
				\arrow{r}{\widehat d}
				\arrow[swap]{d}{1\otimes (-)}
			& \underbrace{R\{\textsc{\textbf{r}}_0, \ldots, \textsc{\textbf{r}}_m\}}_{\widehat C_2}
				\arrow{r}{\widehat d}
				\arrow[swap]{d}{1\otimes (-)}
			& \underbrace{R\{\textsc{\textbf{x}}_1, \ldots, \textsc{\textbf{a}}_m\}}_{\widehat C_1}
				\arrow{r}{\widehat d}
				\arrow[swap]{d}{1\otimes (-)}
			& \underbrace{R}_{\widehat C_0}
				\arrow[swap]{d}{1\otimes (-)}
			\\
			\underbrace{\widehat \Z\{\textbf{z}_1, \ldots, \textbf{z}_m\}}_{\widehat D_3} 
				\arrow{r}{0}
			& \underbrace{\widehat \Z\{\textbf{r}_0, \ldots, \textbf{r}_m\}}_{\widehat D_2}
				\arrow{r}{\widehat d}
			& \underbrace{\widehat \Z\{\textbf{x}_1, \ldots, \textbf{a}_m\}}_{\widehat D_1}
				\arrow{r}{0}
			& \underbrace{\widehat \Z}_{\widehat D_0}
		\end{tikzcd}
	\end{equation}
	
	Applying the functor $\hom_{\widehat \Z}(-, M)$ to the bottom row of diagram~(\ref{eq:profinite_after_trivial_action}) gives us chain complex~(\ref{eq:profinite_after_hom}).
	\begin{equation}\label{eq:profinite_after_hom}
		\begin{tikzcd}[column sep=1em]
			\underbrace{M\{\textbf{z}_1^*, \ldots, \textbf{z}_m^*\}}_{\hom(\widehat D_3, M)} & \underbrace{M\{\textbf{r}_0^*, \ldots, \textbf{r}_m^*\}}_{\hom(\widehat D_2, M)}\arrow{l}{0} & \underbrace{M\{\textbf{x}_1^*, \ldots, \textbf{a}_m^*\}}_{\hom(\widehat D_1, M)}\arrow{l}{\widehat d^*} & \underbrace{M}_{\hom(\widehat D_0, M)} \arrow{l}{0}
		\end{tikzcd}
	\end{equation}
	
	Now, the kernel of the map $\hom(\widehat D_2, M) \to \hom(\widehat D_3, M)$
	is the whole $\hom(\widehat D_2, M)$
	and so $\widehat H^2(\widehat \Delta; M)$ is the cokernel of $\widehat d^*$.
	Since all of $\mathbf{x}_i^*$ and $\mathbf{y}_i^*$
	are mapped to zero by $\widehat d^*$
	this cokernel is precisely
	$$\bigoplus_{i=1}^{m}M \cdot \mathbf{r}_i^*
	/ (\mathbf{r}_0^* + p_i \cdot \mathbf{r}_i^*\ \text{for}\ i = 1, \ldots, m).$$
\end{proof}

\begin{proposition}\label{prop:action_on_profinite_H2}
	Let $\Phi \in \aut(M)$ and $\phi\in \aut(\widehat \Delta)$ be automorphisms,
	and $[\zeta] = [\sum_{i=0}^{m} m_i \cdot \mathbf{r}_i^*]$
	be an element of $\widehat H^2(\widehat \Delta; M)$.
	Then
	$$\Phi \cdot [\zeta] = [\sum_{i=0}^{m} \Phi (m_i) \cdot \mathbf{r}_i^*]$$
	and
	$$\phi^{-1} \cdot [\zeta] = \phi^*([\zeta]) 
	= \kappa \cdot [m_0 \cdot \mathbf{r}_0^*
	+ \sum m_i \cdot \mathbf{r}_{\sigma(i)}^*],$$
	where $\phi(a_i) =\ ^{g_i}\big(a_{\sigma(i)}^{k_i}\big)$
	and $\kappa \in \widehat \Z^\times$ is such that $\kappa \equiv k_i \mod p_i$.
\end{proposition}

The proof of \cref{prop:action_on_profinite_H2}
is the very same calculation
as the proof of \cref{prop:action_on_discrete_H2}
with the only difference being
that $\widehat \Z^\times$ is now a much larger group
than $\Z^\times = \{\pm 1\}$.

\section{Matrix correspondence}\label{sec:matrix_correspondence}

In this section we introduce the `Matrix correspondence'
-- a conglomerate of a few minor results
and our main tool for studying
the central extensions of $\Z^n$ by nice $2$-orbifold groups.
Essentially, it translates the problem of profinite rigidity among these extensions
to a question regarding multiplying matrices and permuting columns.

\begin{proposition}[Matrix correspondence]\label{thm:matrix_correspondence}
	Let $\Delta$ be a nice $2$-orbifold group with $m$ cone points of order $p_1, \ldots, p_m$.
	\begin{enumerate}
		\item There is a $\GL_n(\Z) \times \aut(\Delta)$-equivariant canonical injective homomorphism
		\begin{equation*}
			\Psi\colon H^2(\Delta; \Z^n) \to \varprojlim_{N \in \N} \widehat H^2\big(\widehat \Delta; (\Z/N)^n\big)
		\end{equation*}
		given by sending a cohomology class $[\zeta]$
		representing an equivalence class of central $\Z^n$-by-$\Delta$ extensions
		$$[0 \to \Z^n \to G \to \Delta \to 1]$$
		to the consistent system $([\zeta_N])_N$ of cohomology classes
		representing the equivalence classes of corresponding extensions
		$$[0 \to (\Z/N)^n \to \widehat G/N\widehat \Z^n \to \widehat \Delta \to 1].$$
		
		\item Cohomology classes $[\zeta_1]$ and $[\zeta_2]$ in $H^2(\Delta; \Z^n)$
		represent extension classes $[0 \to \Z^n \to G_i \to \Delta \to 1]$
		with $G_1 \cong G_2$
		\Iff\ they lie in the same orbit of
		$\GL_n(\Z) \times \aut(\Delta) \actson H^2(\Delta; \Z^n)$.
		
		Furthermore, $\widehat G_1 \cong \widehat G_2$
		\Iff\ $\Psi([\zeta_1])$ and $\Psi([\zeta_2])$ lie in the same orbit of
		$\GL_n(\widehat \Z) \times \aut(\widehat \Delta)
		\actson \varprojlim \widehat H^2\big(\widehat\Delta; (\Z/N)^n\big)$.
		
		\item We can identify $H^2(\Delta; \Z^n)$
		with a quotient of $\Z^{n\times(m+1)}$ as follows.
		$$A_\Delta^{(n)} :=
		\Z^{n \times (m+1)} / (\mathbf{r}_0^* + p_i\mathbf{r}_i^*\ \text{for}\ i=1,\ldots, m)$$
		where $\mathbf{r}_0^*, \mathbf{r}_1^*, \ldots, \mathbf{r}_m^*$
		denote generators of $\Z^{n\times(m+1)}$
		treated as a free $\Z^n$-~module,
		where $x\cdot \mathbf{r}_i^*$ will denote
		the matrix having vector $x$ as the $i$-th column
		and zeros everywhere else.
		
		Similarly, we can identify $\varprojlim \widehat H^2\big(\widehat\Delta; (\Z/N)^n\big)$
		with $\widehat A_\Delta^{(n)}$.
		
		Furthermore, the orbits of the actions
		$$\GL_n(\Z) \times \aut(\Delta) \actson A_\Delta^{(n)}
		\quad \text{and} \quad
		\GL_n(\widehat \Z) \times \aut(\widehat \Delta) \actson \widehat A_\Delta^{(n)}$$
		are the same as those of  $\GL_n(\Z) \times \Sigma \actson A_\Delta^{(n)}$
		and those of $\GL_n(\widehat\Z) \times \Sigma \actson \widehat A_\Delta^{(n)}$,
		where $$\Sigma = \{\sigma \in \Sym(m)\ |\ p_{\sigma(i)} = p_i\ \text{for}\ i=1, \ldots, m\}$$
		acts by permuting the columns by $\sigma(\mathbf{r}_i^*) = \mathbf{r}_{\sigma(i)}^*$.
	\end{enumerate}
\end{proposition}

\begin{proof}[Proof of part 1.]
	The map $m_N: H^2(\Delta; \Z^n) \to H^2\big(\Delta; (\Z/N)^n\big)$
	coming from reduction of the coefficients modulo $N$
	is canonical and thus equivariant under the action of
	$\GL_n(\Z) \times \aut(\Delta)$.
	Now, by goodness of $\Delta$, the canonical map
	$$\iota_N: \widehat H^2\big(\widehat \Delta; (\Z/N)^n\big) \to H^2\big(\Delta; (\Z/N)^n\big)$$
	is an isomorphism, so we can form the composition
	$$\iota_N^{-1}\circ m_N: H^2(\Delta; \Z^n) \to \widehat H^2\big(\widehat\Delta; (\Z/N)^n\big).$$
	These maps are canonical and thus consistent with the transition maps
	$$m_{N_1N_2}: \widehat H\big(\widehat \Delta; (\Z/N_1)^n\big)
	\to \widehat H^2\big(\widehat \Delta; (\Z/N_2)^n\big),$$
	thus giving a canonical map to $\varprojlim \widehat H^2\big(\widehat\Delta; (\Z/N)^n\big)$.
	
	The maps $m_N$ and $\iota_N$ correspond to the following maps of extensions
	and thus the composition $\iota_N^{-1}\circ m_N$ maps $[\zeta]$ to $[\zeta_N]$.
	$$\begin{tikzcd}
		0 \arrow{r}
			& \Z^n \arrow{r} \arrow{d}{\mod N}
				& G \arrow{r} \arrow{d}
					& \Delta \arrow{r} \arrow[equal]{d}
						& 1 \\
		0 \arrow{r}
			& (\Z/N)^n \arrow{r} \arrow[equal]{d}
				& G / N\Z^n \arrow{r} \arrow{d}{\mathsf{h}_{G / N\Z^n}}
					& \Delta \arrow{r} \arrow{d}{\mathsf{h}_\Delta}
						& 1 \\
		0 \arrow{r}
		& (\Z/N)^n \arrow{r}
			& \widehat G / N\widehat \Z^n \arrow{r}
				& \widehat \Delta \arrow{r}
					& 1 \\
	\end{tikzcd}$$
\end{proof}

\begin{proof}[Proof of part 2.]
	By \cref{prop:iso_centerless},
	the isomorphism classes of groups $G$ with centre $\mathcal{Z}(G) \cong\Z^n$
	and quotient $G/\mathcal{Z}(G) \cong \Delta$
	are in bijection with the set $\overline{\mathcal{E}}(\Delta; \Z^n)$
	of similarity classes of central extensions of $\Z^n$ by $\Delta$.
	Also, this set is in bijection with the set of orbits
	of the action $\GL_n(\Z) \times \aut(\Delta) \actson H^2(\Delta; \Z^n)$.
	Thus $G_1 \cong G_2$ \Iff\ the cocycles $[\zeta_1]$ and $[\zeta_2]$
	lie in the same orbit.
	
	Furthermore, since $\widehat G_i = \varprojlim \widehat G_i / N\widehat \Z^n$,
	a consistent set of isomorphisms
	${\tilde\phi_N: \widehat G_1 / N\widehat \Z^n \to \widehat G_2 / N\widehat \Z^n}$
	gives an isomorphism $\tilde\phi: \widehat G_1 \to \widehat G_2$.
	If $(\widehat\Phi, \phi) \in \GL_n(\widehat\Z) \times \aut(\Delta)$
	is such that $(\widehat\Phi, \phi) \cdot [\zeta_{1,N}] = [\zeta_{2,N}]$
	then this gives a consistent set of isomorphisms $(\tilde\phi_N)$, so $\widehat G_1 \cong \widehat G_2$.
	
	Conversely, if $\tilde\phi: \widehat G_1 \to \widehat G_2$ is an isomorphism
	then it must give an automorphism $\widehat\Phi$ of the kernel $\widehat\Z^n$
	and an automorphism $\phi$ of the quotient $\widehat \Delta$.
	Then clearly $(\widehat\Phi, \phi)\cdot [\zeta_{1,N}] = [\zeta_{2,N}]$ for all $N$
	and so indeed $\Psi([\zeta_1])$ and $\Psi([\zeta_2])$
	lie in the same orbit of the action
	$$\GL_n(\widehat \Z) \times \aut(\widehat \Delta)
	\actson \varprojlim \widehat H^2\big(\widehat\Delta; (\Z/N)^n\big).$$
\end{proof}

\begin{proof}[Proof of part 3.]
	\Cref{prop:H2_of_nice_orbifolds,prop:H2_hat_of_nice_2-orbifolds}
	computed the second cohomology group $H^2(\Delta; M)$
	for trivial $\Delta$-modules $M$
	and $\widehat H^2(\widehat \Delta; M)$
	for finite trivial $\widehat \Delta$-modules $M$.
	Thus we indeed have $H^2(\Delta; \Z^n) \cong A_\Delta^{(n)}$
	via a natural isomorphism.
	Similar computation gives a canonical isomorphism
	$\varprojlim \widehat H^2\big(\widehat \Delta; (\Z/N)^n\big) \cong \widehat A_\Delta^{(n)}$
	and the map $\Psi$ can be identified with
	$\mathsf{h}: A_\Delta^{(n)} \to \widehat A_\Delta^{(n)}$.
	
	Finally, we need to use the explicit formulas from
	\cref{prop:action_on_discrete_H2,prop:action_on_profinite_H2}
	for the actions of $\aut(\Delta)$ and $\aut(\widehat\Delta)$
	to show that the orbits don't change when we restrict the actions
	to smaller subgroups---respectively
	$\GL_n(\Z) \times \Sigma$ and $\GL_n(\widehat \Z) \times \Sigma$.
	An element $\phi^{-1} \in \aut(\Delta)$ acts by
	$$\phi^{-1} \cdot [\zeta] = \phi^*([\zeta])
	= \kappa \cdot [m_0 \cdot \mathbf{r}_0^*
	+ \sum_{i=1}^m m_i \cdot \mathbf{r}_{\sigma(i)}^*],$$
	where $\phi(a_i) =\ ^{g_i}\big(a_{\sigma(i)}^{k_i}\big)$
	and $\kappa \in \Z^\times$ is such that $\kappa \equiv k_i \mod p_i$.
	Thus, for all $[\zeta]\in H^2(\Delta; \Z^n)$ we have
	$(\id, \phi^{-1}) \cdot [\zeta] = (\kappa\cdot\id, \sigma) \cdot [\zeta]$,
	and thus the orbits of $\GL_n(\Z) \times \aut(\Delta)$
	and of $\GL_n(\widehat\Z)\times \Sigma$
	are indeed the same.
	Similarly, an automorphism $\phi^{-1} \in\aut(\widehat \Delta)$ acts by the very same formula
	with the exception that $\kappa$ is now allowed to be in a much bigger group $\widehat\Z^\times$.
	Finally, $\kappa \cdot \id$ is an element of $\GL_n(\widehat \Z)$
	and thus the orbits of $\GL_n(\widehat\Z) \times \aut(\widehat \Delta)$
	are the same as those of $\GL_n(\widehat \Z) \times \Sigma$.
\end{proof}

\section{Using the matrix correspondence}\label{sec:matrix_correspondence_implications}

This section aims at understanding the group $A_\Delta^{(n)}$
introduced in \cref{thm:matrix_correspondence}
together with the actions of $\GL_n(\Z)$ and $\Sigma$ on it.
The question we investigate is:
how are the orbits of ${\GL_n(\Z) \times \Sigma \actson A_\Delta^{(n)}}$
related to the orbits of ${\GL_n(\widehat\Z) \times \Sigma \actson \widehat A_\Delta^{(n)}}$?

Throughout this section $\Delta$ will be any nice $2$-orbifold group with cone points of orders $p_1, \ldots, p_m$. Let $A_\Delta^{(n)}$ be as defined in \cref{thm:matrix_correspondence}, i.e. a quotient of the additive group $\Z^{n\times (m+1)}$ of $n\times(m+1)$ matrices with integer coefficients by relations $\mathbf{r}_0^* + p_i\mathbf{r}_i^*$ for $i=1, \ldots, m$ on columns.

\subsection{General structure of $A_\Delta^{(n)}$}

The next proposition enables us to understand better
the structure of $A_\Delta^{(n)}$
and the action of $\GL_n(\Z) \times \Sigma$ on it.

\begin{proposition}\label{prop:quotient_of_A}
	The projection $\Z^{n\times(m+1)} \to \bigoplus_{i=1}^m (\Z/p_i)^n$ descends to a short exact sequence
	$$0 \to \Z^n \to A_\Delta^{(n)} \xrightarrow{q} \bigoplus (\Z/p_i)^n \to 0,$$ which is equivariant under the action of $\GL_n(\Z) \times \Sigma$. Also, $\Sigma$ acts trivially on the kernel $\Z^n = \ker\big(q\colon A_\Delta^{(n)} \to \bigoplus(\Z/p_i)^n\big)$.
\end{proposition}

\begin{proof}
	To make following the proof easier, all mentioned groups and maps between them are shown in diagram~(\ref{eq:huge_diagram}).
	
	All of the relations $\mathbf{r}_0^*+p_i\mathbf{r}_i^*$ are sent to $0$ by the projection map, so we indeed get a map $A_\Delta^{(n)} \twoheadrightarrow \bigoplus(\Z/p_i)^n$. The kernel of this map is the quotient of $\Z^n\oplus \bigoplus p_i\Z^n$ by the kernel
	$$\Z^n\{\mathbf{r}_0^*+p_i\mathbf{r}_i^*\ |\ i=1, \ldots, m\}
	= \ker\big(\Z^{n\times(m+1)} \twoheadrightarrow A_\Delta^{(n)}\big),$$
	which is isomorphic to $\Z^n$.
	
	\begin{equation}\label{eq:huge_diagram}
		\begin{tikzcd}[column sep=1.5em]
			& 0 \arrow{d}	& 0 \arrow{d} \\
			& \Z^n\{\mathbf{r}_0^* + p_i\mathbf{r}_i^*\} \arrow[equal]{r} \arrow{d}	& \Z^n\{\mathbf{r}_0^* + p_i\mathbf{r}_i^*\} \arrow{d}	&  \\
			0 \arrow{r}	& \Z^n \oplus \bigoplus_{i=1}^m p_i \Z^n \arrow{r} \arrow{d}	& \Z^{n\times(m+1)} \arrow{r} \arrow{d}	& \bigoplus_{i=1}^{m} (\Z/p_i)^n \arrow{r} \arrow[equal]{d}	& 0
			\\ 
			0 \arrow{r}
			& \Z^n \arrow{r} \arrow{d}
			& A_\Delta^{(n)} \arrow{r}{q} \arrow{d}
			& \bigoplus_{i=1}^{m} (\Z/p_i)^n \arrow{r}							& 0
			\\
			& 0	& 0	& &
		\end{tikzcd}
	\end{equation}
\end{proof}

%The elements of $\Z^n = \ker q$ have a role similar to the \emph{Euler class} of a Seifert Fibred Space.

The next result actually computes the group $A_\Delta^{(n)}$.

\begin{proposition}\label{prop:Smith_form_for_A}
	There is an isomorphism
	$$\bigoplus_{i=1}^m \big(\Z/p_i\big)^n
	\xrightarrow{\alpha} \bigoplus_{j=1}^m \big(\Z/d_j\big)^n$$
	where $d_j$ are such that
	$$\begin{pmatrix}
		d_1	& 0		& \ldots & 0 \\
		0	& d_2	& \ldots & 0 \\
		\vdots	& \vdots	& \ddots	& \vdots \\
		0	& 0		& \ldots & d_m
	\end{pmatrix}\
	\text{is the Smith normal form of}\
	\begin{pmatrix}
		p_1	& 0		& \ldots & 0 \\
		0	& p_2	& \ldots & 0 \\
		\vdots	& \vdots	& \ddots	& \vdots \\
		0	& 0		& \ldots & p_m
	\end{pmatrix}$$
	i.e.
	\begin{equation}\label{eq:Smith_exact_form}
		d_j = \frac{\gcd(p_{i_1}p_{i_2}\ldots p_{i_j}\ \text{for}\ 1\le i_1 < i_2 < \ldots < i_j \le m)}{\gcd(p_{i_1}p_{i_2}\ldots p_{i_{j-1}}\ \text{for}\ 1\le i_1 < i_2 < \ldots < i_{j-1} \le m)}.
	\end{equation}
	
	Then there is also an isomorphism
	$$A_\Delta^{(n)} \xrightarrow{\beta} \Z^n \oplus \bigoplus_{j=1}^{m-1} \big(\Z/d_j)^n.$$
	Furthermore, $\alpha$ and $\beta$ are equivariant
	with respect to the action of $\GL_n(\Z)$
	and this identification gives the following commutative diagram.
	\begin{equation*}
		\begin{tikzcd}[column sep=1.5em]
			0 \arrow{r}
			& \Z^n \arrow{r} \arrow[equal]{d}
			& A_\Delta^{(n)} \arrow{r}{q} \arrow{d}{\beta}
			& \bigoplus_{i=1}^m \big(\Z/p_i\big)^n \arrow{r} \arrow{d}{\alpha}	& 0
			\\
			0 \arrow{r}
			& \Z^n \arrow{r}{\cdot d_m}
			& \Z^n \oplus \bigoplus_{j=1}^{m-1} \big(\Z/d_j)^n \arrow{r}
			& \bigoplus_{j=1}^m \big(\Z/d_j\big)^n \arrow{r}
			& 0
		\end{tikzcd}
	\end{equation*}
\end{proposition}

\begin{proof}
	The fact that $\alpha$ is an isomorphism
	and that $d_j$ are given by the formula~\ref{eq:Smith_exact_form}
	comes from computation of the Smith normal form using matrix minors.
	Now,
	$A_\Delta^{(n)}$ is isomorphic to
	$\Z^n \otimes_{\Z} \coker\big(\Z^m \to \Z^{m+1}\big)$
	where the map is given by the matrix
	\begin{equation*}
		\begin{pmatrix}
			1	& 1		& \ldots	& 1 \\
			p_1	& 0		& \ldots	& 0 \\
			0	& p_2	& \ldots	& 0 \\
			\vdots	& \vdots		& \ddots	& \vdots \\
			0	& 0		& \ldots	& p_m
		\end{pmatrix}.
	\end{equation*}
	We can first use the row and column operations
	which gave us $\alpha$ to get the diagonal matrix
	with $d_j$ for $j=1, 2, \ldots, m$
	and later do the following column operations:
	\begin{equation*}
		\begin{pmatrix}
			1		& 1			& \ldots	& 1			& 1 \\
			d_1		& 0			& \ldots	& 0			& 0 \\
			0		& d_2		& \ldots	& 0			& 0 \\
			\vdots	& \vdots	& \ddots 	& \vdots	& \vdots \\
			0		& 0			& \ldots	& d_{m-1}	& 0 \\
			0		& 0			& \ldots	& 0 		& d_m
		\end{pmatrix}
		\mapsto
		\begin{pmatrix}
			0		& 0			& \ldots	& 0			& 1 \\
			d_1		& 0			& \ldots	& 0			& 0 \\
			0		& d_2		& \ldots	& 0			& 0 \\
			\vdots	& \vdots	& \ddots 	& \vdots	& \vdots \\
			0		& 0			& \ldots	& d_{m-1}	& 0 \\
			-d_m	& -d_m		& \ldots	& -d_m 		& d_m
		\end{pmatrix}.
	\end{equation*}
	Now, since $d_j\ |\ d_m$ for all $j=1, 2, \ldots, m$
	we can fully eliminate the bottom row
	with row operations affecting only the bottom row,
	which means that the first $m$ basis elements for $\Z^{m+1}$
	weren't changed,
	whence that the diagram in the statement
	actually commutes.
	
	Finally, because we only manipulated the $\Z^n$-bases,
	the maps are indeed $\GL_n(\Z)$-equivariant.
\end{proof}

\Cref{prop:quotient_of_A} allowed us to understand better
the torsion part of $A_\Delta^{(n)}$,
but it will be necessary to understand
the torsion-free part too.
\Cref{prop:Euler_class} helps with this.
It comes from the notion of \emph{Euler class}
defined for Seifert Fibered Spaces.

\begin{proposition}\label{prop:Euler_class}
	Let $\tilde E\colon  \Z^{n\times(m+1)} \to \Z^n$ be defined by
	$$\tilde E\colon  \sum_{i=0}^m x_i \cdot \mathbf{r}_i^*
	\mapsto -d_m \cdot x_0 + \sum_{i=1}^m \frac{d_m}{p_i} \cdot x_i.$$
	Then $\tilde E$ factors through
	$\Z^{n\times(m+1)} \twoheadrightarrow A_\Delta^{(n)}$
	giving $E\colon A_\Delta^{(n)} \to \Z^n$.
	Furthermore, the map is surjective and equivariant
	with respect to the action of $\GL_n(\Z) \times \Sigma$
	(with $\Sigma$ acting trivially).
\end{proposition}

\begin{proof}
	First of all, notice that
	\begin{align*}
		d_m  & = {p_1\ldots p_m}\
		/\ {\gcd(p_1\ldots p_{m-1},\ p_1\ldots p_{m-2}p_m,\ \ldots,\ p_2p_3\ldots p_m)} \\
		& = \lcm(p_1, \ldots, p_m),
	\end{align*}
	so all of the $\frac{d_m}{p_i}$ are integers
	and the map $\tilde E\colon  \Z^{n\times(m+1)} \to \Z^n$ is actually well-defined.
	All of the relations are sent to zero by the map
	$$E(x\cdot \mathbf{r}_0^* + p_i \cdot x \cdot \mathbf{r}_i^*)
	= - d_m \cdot x + \frac{d_m}{p_i}\cdot p_i\cdot x = 0,$$
	so $\tilde E$ indeed factors through $\Z^{n(m+1)} \twoheadrightarrow A_\Delta^{(n)}$
	giving $E\colon A_\Delta^{(n)} \to \Z^n$.
	$\tilde E$ is clearly equivariant \wrt\ the action of $\GL_n(\Z)$ and also \wrt\ the action of $\Sigma$ since elements of $\Sigma$ permute only the columns which have the same $p_i$. Since the quotient $\Z^{n\times(m+1)} \twoheadrightarrow A_\Delta^{(n)}$ is also equivariant, the induced map $E$ is equivariant too.
	Now, to get surjectivity let's notice that $\gcd(d_m/p_1, \ldots, d_m/p_m) = 1$ so there are integers $k_i$ such that $\sum k_i \cdot \frac{d_m}{p_i} = 1$. Now this means that
	$$E\big(\sum_{i=1}^m k_i\cdot v\cdot \mathbf{r}_i^*\big) = v$$
	for any $v \in \Z^n$, so the map is indeed surjective.
\end{proof}

\subsection{Orbits collapsing}

Using the quotient map shown in \cref{prop:quotient_of_A}
we can already show that some orbits of
${\GL_n(\Z) \times \Sigma \actson A_\Delta^{(n)}}$
become identified in $\GL_n(\widehat\Z) \times \Sigma \actson \widehat A_\Delta^{(n)}$.
This is the content of \cref{prop:examples_for_lack_of_rigidity},
but before we can prove it,
we need \cref{lem:profinite_matrix_works}
to know when the effect of acting
by a profinite matrix $\hat\Phi \in \GL_n(\widehat\Z)$
on some $[A] \in A_\Delta^{(n)} < \widehat A_\Delta^{(n)}$
gives a matrix $\hat\Phi\cdot [A] \in \widehat A_\Delta^{(n)}$
which still lies in the copy of $A_\Delta^{(n)}$.

\begin{lemma}\label{lem:profinite_matrix_works}
	Let $[A] \in A_\Delta^{(n)}$
	and let $\hat\Phi$ be a profinite matrix in $\GL_n(\widehat\Z)$.
	Then $\hat\Phi \cdot [A]$ is in $A_\Delta^{(n)}$
	\Iff\ $\hat\Phi \cdot E([A]) \in \Z^n$
	for the map $E \colon  A_\Delta^{(n)} \to \Z^n$ defined in \cref{prop:Euler_class}.
\end{lemma}

\begin{proof}
	Consider the composition
	$$(E \circ \beta^{-1}) \colon  \Z^n \oplus
	\bigoplus_{j=1}^{m-1} \big(\Z/d_j\big)^n \twoheadrightarrow \Z^n$$
	as shown in diagram~(\ref{eq:profinite_matrix_works}). Its surjectivity implies that the kernel must be precisely the torsion subgroup $\bigoplus_{j=1}^{m-1}\big(\Z/d_j\big)^n$. Thus we can regard $E \circ \beta^{-1}$ as a projection onto the $\Z^n$ direct factor with possibly some automorphism of $\Z^n$.
	
	Now, the bottom row is equivariant with respect to the action of $\GL_n(\widehat\Z)$,
	so $$\hat\Phi \cdot E([A]) \in \Z^n \Longrightarrow \widehat E(\hat\Phi\cdot[A]) \in \Z^n,$$
	which means that
	the $\widehat\Z^n$-coordinate of $\hat\Phi\cdot \beta^{-1}([A])$
	is actually in $\Z^n$,
	but the set of all such elements is precisely the $\beta(A_\Delta^{(n)})$
	inside $\widehat\beta(\widehat A_\Delta^{(n)})$.
	\begin{equation}\label{eq:profinite_matrix_works}
		\begin{tikzcd}
			\Z^n \oplus \bigoplus_{j=1}^{m-1} \big(\Z/d_j\big)^n \arrow{r}{\beta^{-1}}[swap]{\cong} \arrow{d}
			& A_\Delta^{(n)} \arrow[two heads]{r}{E} \arrow{d}
			& \Z^n \arrow{d} \\
			\widehat\Z^n \oplus \bigoplus_{j=1}^{m-1} \big(\Z/d_j\big)^n \arrow{r}{\widehat\beta^{-1}}[swap]{\cong}
			& \widehat A_\Delta^{(n)} \arrow[two heads]{r}{\widehat E}
			& \widehat\Z^n
		\end{tikzcd}
	\end{equation}
\end{proof}

\begin{proposition}\label{prop:examples_for_lack_of_rigidity}
	Let's assume that $n > 1$
	and that there exists a natural number $d$ different from $1, 2, 3, 4$ and $6$,
	such that at least $n$ of the numbers $p_i$ are divisible by $d$.
	Then there exist $[A], [B] \in A_\Delta^{(n)}$ and a matrix $\hat\Phi \in \GL_n(\widehat\Z)$
	such that $\hat\Phi \cdot [A] = [B]$,
	but there is no integer matrix $\Phi \in \GL_n(\Z)$ and $\sigma \in \Sigma$
	such that $\Phi\cdot[A] = \sigma \cdot [B]$.
\end{proposition}

\begin{proof}
	In \cref{prop:quotient_of_A} we defined the map
	$q\colon A_\Delta^{(n)} \twoheadrightarrow \bigoplus_{i=1}\big(\Z/p_i\big)^n$.
	Setting $l$ to be the number of $p_i$ divisible by $d$,
	we can quotient further to get
	\begin{equation*}
		\begin{tikzcd}
			\Z^{n\times(m+1)} \arrow[two heads]{r}{p}
			& A_\Delta^{(n)} \arrow[two heads]{r}{q}
			& \bigoplus_{i=1}\big(\Z/p_i\big)^n \arrow[two heads]{r}{q'}
			& \bigoplus_{i=1}^l (\Z/d)^n
		\end{tikzcd}.
	\end{equation*}
	The map $q'$ is clearly equivariant with respect to the action of $\GL_n(\Z)$
	and since $\Sigma$ can permute columns without changing their respective $p_i$,
	we get a quotient action of $\Sigma$ on $\bigoplus_{i=1}^l \big(\Z/d\big)^n$
	by column permutation.
	
	We can assume that the $p_i$ which are divisible by $d$ are $p_1, \ldots, p_l$.
	Now let $a$ be an integer coprime to $d$ and such that $a \not \equiv \pm 1 \mod p_n$%
	---here is where we use that $d \ne 1, 2, 3, 4, 6$ and that $d\ |\ p_n$---%
	and consider the following $n\times (m+1)$ matrix $A \in \Z^{n\times(m+1)}$.
	\begin{equation*}
		A = \begin{pmatrix}
			0	& 1		& \dots		& 0 		& 0 		& 0 		& \dots 	& 0 \\
			\vdots &\vdots 	& \ddots 	& \vdots 	& \vdots 	& \vdots 	& \ddots 	& \vdots \\
			0	& 0		& \dots 	& 1 		& 0 		& 0 		& \dots 	& 0 \\
			0	& 0		& \dots 	& 0 		& 1 		& 0 		& \dots 	& 0 
		\end{pmatrix}
	\end{equation*}
	
	Now, take any matrix $\hat\Phi \in \GL_n(\widehat\Z)$ such that
	\begin{itemize}
		\item $\hat\Phi \cdot E([A]) = E([A])$ -- in particular, $\hat\Phi \cdot [A] \in A_\Delta^{(n)}$;
		\item $\det \hat\Phi \not\equiv \pm 1 \mod d$.
	\end{itemize}
	We set $[B] := \hat\Phi \cdot [A]$
	and show that there is no $\Phi \in \GL_n(\Z)$ and $\sigma \in \Sigma$
	such that
	$$\Phi \cdot (q'\circ q)([A]) = \sigma \cdot (q'\circ q)([\hat\Phi\cdot [A]]) = \sigma \cdot (q'\circ q)([B]).$$
	
	$(q'\circ q)([A])$ is the $n\times n$ identity matrix
	with $(n-l)$ zero columns appended on the right.
	Thus $\Phi \cdot (q'\circ q)([A])$ is the matrix $(\Phi \mod d)$
	appended by $(n-l)$ zero columns.
	Note that $(\Phi \mod d)$ can't have any zero columns.
	In a similar spirit $(q'\circ q)(\hat\Phi\cdot[A])$
	is the matrix $(\hat\Phi \mod d)$%
	---whose columns are all non-zero---%
	followed by $(n-l)$ zero columns.
	This implies that $\sigma$ sends the last $(n-l)$ columns to themselves
	and possibly permutes the first $n$ ones.
	
	Thus we get that $\Phi \equiv \sigma \cdot \hat\Phi \mod d$,
	which in turn implies that ${\det \Phi \equiv \det \sigma \cdot \det \hat \Phi}$,
	but $\det \Phi \equiv \pm 1$ and $\det \sigma \equiv \pm 1$,
	while $\hat\Phi$ was chosen specifically so that ${\det \hat\Phi \not \equiv \pm 1}$.
\end{proof}

\subsection{Orbits separated}

Analysing the action of $\GL_n(\widehat\Z)$ closely
we will show that sometimes orbits don't change.

\begin{lemma}\label{lem:upper_triangular}
	Given any matrix $A \in \Z^{n\times(m+1)}$,
	we can find a matrix $\Phi \in \GL_n(\Z)$
	such that $A'= \Phi A$ is an upper-triangular matrix,
	i.e. that $A'_{ij} = 0$ for $i > j$.
\end{lemma}

\begin{proof}
	We can multiply by permutation matrices from $\GL_n(\Z)$
	to permute rows of $A$
	and multiply by elementary matrices to perform elementary row operations
	(adding an integer multiple of a row to another one).
	Following Euclid's algorithm,
	we can in this way make the first column have at most one non-zero entry%
	---equal to the $\gcd$ of its entries---in the top place.
	We can do the same for the next column,
	not doing any operations that would change the first row.
	Proceeding inductively, we get an upper-triangular matrix.
\end{proof}

In the following $\SL_n^{\pm 1}(R)$ denotes
the group of $n\times n$ matrices with coefficients in $R$
and determinant equal to $\pm 1$.

\begin{lemma}\label{lem:image_of_GLZ}
	In the following commuting diagram,
	with $v_1$ and $v_2$ being reductions modulo $d$
	\begin{equation*}
		\begin{tikzcd}
			\GL_n(\Z) \arrow{r}\arrow[swap]{rd}{v_1} & \SL_n^{\pm 1} (\widehat{\Z}) \arrow{d}{v_2} \\
			& \GL_n(\Z/d)
		\end{tikzcd}
	\end{equation*}
	the image of $v_1$ is the same as the image of $v_2$.
	In particular, given a matrix $\hat \Phi$ with profinite entries
	whose determinant is $\pm 1$,
	we can find an integer matrix $\Phi$ with determinant $\pm 1$
	\st\ $\hat \Phi \equiv \Phi \mod d$.
	
	Furthermore, if the first column of $\hat\Phi$ is $(\pm 1,0,\ldots, 0)^T$,
	then we can require that the first column of $\Phi$
	is also $(\pm 1,0,\ldots, 0)^T$ (with the same choice of sign).
\end{lemma}

\begin{proof}
	The image of $v_2\colon \SL_n^{\pm 1}(\widehat\Z) \to \GL_n(\Z/d)$
	is inside $\SL_n^{\pm 1}(\Z / d)$,
	and the image of $v_1\colon \GL_n(\Z) \to \GL_n(\Z/d)$ is $\SL_n^{\pm 1}(\Z / d)$,
	so two images must be the same.
	
	For the second part, let's notice that
	if $\hat\Phi_{1*} = (\pm 1, 0, \ldots, 0)^T$,
	then $\hat\Phi $ is a block matrix
	$\begin{pmatrix}
		\pm 1 & (\hat{\textbf{v}}')^T \\
		\textbf{0} & \hat\Phi'
	\end{pmatrix}$
	where $\det \hat\Phi' = \pm 1$.
	Then we take $\Phi' \equiv \hat\Phi'$
	and any $\textbf{v}' \equiv \hat{\textbf{v}}'$ modulo $d$ and build
	$$\Phi = \begin{pmatrix}
		\pm 1 & (\textbf{v}')^T \\
		\textbf{0} & \Phi'
	\end{pmatrix}.$$
\end{proof}

Having proved \cref{lem:upper_triangular} and \cref{lem:image_of_GLZ},
we can prove our main result about orbits remaining distinct.

The reason that the statement counts $j$ such that $d_j \not \in \{1, 2, 3, 4, 6\}$
is that $1, 2, 3, 4, 6$ are the only numbers $t$
such that $(\Z/t)^\times$ has at most two elements.

\begin{proposition}\label{prop:profinite_matrix_multiplication_large_n}
	Let $\Delta$ be a nice $2$-orbifold with $m \ge 0$ cone points
	of order $p_1, \ldots, p_m$.
	Let $A_\Delta^{(n)} \cong \Z^n \oplus \bigoplus_{j=1}^{m-1} \big(\Z/d_j)^n$ as in \cref{prop:Smith_form_for_A}
	and let $$k = 1 + |\big\{1 \le j \le m-1\ |\ d_j \not \in \{1, 2, 3, 4, 6\}\big\}|.$$
	Let $[A], [B] \in A_\Delta^{(n)}$ be related by $\hat\Phi \cdot [A] = [B]$ for some $\hat\Phi\in\GL_n(\widehat\Z)$.
	
	Then, if $n > k$, we can find a matrix $\Phi\in \GL_n(\Z)$ with integer coefficients
	such that $\Phi\cdot [A] = \hat\Phi \cdot [A] = [B]$.
\end{proposition}

\begin{proof}
	First, let's take $A$ and $B$ to their upper-triangular form
	with respect to the decomposition
	$A_\Delta^{(n)} \cong \Z^n \oplus \bigoplus_{j=1}^{m-1} \big(\Z/d_{m-j})^n$
	-- note that here we wrote $A_\Delta^{(n)}$ as a quotient of
	$\Z^n\cdot \mathbf{s}_0^* \oplus \Z^n \cdot \mathbf{s}_{m-1}^* \oplus \ldots
	\oplus  \Z^n \cdot \mathbf{s}_1^*$
	and the matrices $A'$ and $B'$ were chosen to be upper-triangular
	with respect to this basis
	$(\mathbf{s}_0^*, \mathbf{s}_{m-1}^*, \ldots, \mathbf{s}_1^*)$.
	Let $\Phi_A, \Phi_B \in \GL_n(\Z)$ be the matrices
	used to perform this upper-triangulisation,
	i.e. such that $\Phi_A \cdot A$ and $\Phi_B \cdot B$ are upper-triangular.
	With that we get
	\begin{align*}
		\hat\Phi \cdot [A] = [B]
		& \Longleftrightarrow \big(\Phi_B\ \hat\Phi\ \Phi_A^{-1}\big)
		\cdot [\Phi_A \cdot A]= [\Phi_B \cdot B] \\
		& \Longleftrightarrow \hat\Phi' \cdot [A'] = [B'],
	\end{align*}
	where $\hat\Phi' = \big(\Phi_B^{-1} \hat\Phi\ \Phi_A^{-1}\big),\ A' = \Phi_A\cdot A$ and $B' = \Phi_B \cdot B$.
	
	Now, let
	$A' = x_0 \cdot \mathbf{s}_0^* + x_{m-1}\cdot \mathbf{s}_{m-1}^*
	+ \ldots + x_1 \cdot \mathbf{s}_1^*$
	and similarly
	$B' = y_0 \cdot \mathbf{s}_0^* + y_{m-1}\cdot \mathbf{s}_{m-1}^*
	+ \ldots + y_1 \cdot \mathbf{s}_1^*$.
	Then,
	\begin{equation*}
		\hat\Phi' \cdot [A'] = [B'] \Longleftrightarrow \hat\Phi' x_0 = y_0,
		\hat\Phi' x_j \equiv y_j\mod d_j\ \text{for}\ j=1, \ldots, m-1.
	\end{equation*}
	
	Now comes the crucial bit of the proof.
	
	Consider $\hat\Phi'' \in \GL_n(\widehat\Z)$ obtained from $\hat\Phi'$
	by multiplying the last column by $\kappa \in \widehat\Z^\times$.
	We claim that if $\kappa \equiv 1 \mod d_{m-(n-1)}$,
	then still $\hat\Phi'' \cdot [A'] = [B']$.
	
	There are two cases to check here.
	\begin{enumerate}
		\item For $j = 0,\ m-1, m-2, \ldots, m-(n-2)$,
		the bottom coordinate of $x_j$ is zero
		(as $A'$ is upper-triangular).
		Then, changing the last column of $\hat\Phi'$
		doesn't affect the result of multiplication;
		in fact $\hat\Phi'' \cdot x_j = y_j$
		with genuine equality, not just congruence.
		\item For $j = m-(n-1), m-n, \ldots, 1$,
		we have $d_j\ |\ d_{m-(n-1)}$ and so
		${\hat\Phi'' \equiv \hat\Phi' \mod d_j}$
		as $\kappa \equiv 1\mod d_{m-(n-1)}$.
		In particular, $\hat\Phi'' \cdot x_j
		\equiv \hat\Phi' \cdot x_j = y_j \mod d_j$.
	\end{enumerate}
	Thus we know that indeed $\hat\Phi''\cdot[A'] = [B']$.
	
	Now, notice that we have some freedom in the choice of $\kappa$.
	We only assumed that $\kappa \in \widehat\Z^\times$
	and that $\kappa \equiv 1\mod d_{m-(n-1)}$.
	Since $d_{m-(n-1)}$ is in the set $\{1, 2, 3, 4, 6\}$,
	the integers coprime to $d_{m-(n-1)}$ are $\pm 1 \mod d_{m-(n-1)}$,
	so in particular
	$\det \hat\Phi' \equiv \pm 1 \mod d_{m-(n-1)}$.
	Since $\det \hat\Phi'' = \kappa \cdot \det\hat\Phi'$,
	we can choose $\kappa = \pm(\det\hat\Phi')^{-1}$
	so that $\det \hat\Phi'' = \pm 1$,
	i.e. $\hat\Phi'' \in \SL_n^{\pm 1}(\widehat\Z)$.
	
	At this point we know that $\hat\Phi'' \cdot [A'] = [B']$
	for some $\hat\Phi'' \in \SL_n^{\pm 1}(\widehat\Z)$.
	By \cref{lem:image_of_GLZ},
	we can choose $\Phi \in \GL_n(\Z)$ such that
	$\Phi \equiv \hat\Phi'' \mod d_{m-1}$
	and additionally
	if the first column of $\hat\Phi''$ is $\pm (1,0,\ldots, 0)^T$,
	then we can require the first column of $\Phi$ to be the same.
	
	Now, since $d_j\ |\ d_{m-1}$ for all $j = 1, \ldots, m-1$,
	we have
	$$\Phi \cdot x_j \equiv \hat\Phi'' \cdot x_j = y_j \mod d_j$$
	for all $j \ne 0$.
	For $j=0$ notice that if $x_0\ne 0$,
	then the first column of $\hat\Phi''$
	must be $\pm (1,0,\ldots, 0)^T$
	and it is the same as the first column of $\Phi$,
	so $\Phi \cdot x_0 = y_0$.
	If $x_0=y_0 = 0$,
	then we get $\Phi \cdot x_0 = y_0$ trivially.
	
	Finally, $\Phi \cdot [A'] = [B']$
	gives $\big(\Phi_B^{-1}\Phi\ \Phi_A\big) \cdot [A] = [B]$.
\end{proof}

We can use the same idea of modifying the last column of $\hat\Phi$
which was used in \cref{prop:profinite_matrix_multiplication_large_n}
to show that in fact
the elements of $A_\Delta^{(n)}$
lying in the same orbit of $\GL_n(\widehat\Z)$
actually become members of the same orbit of
$\GL_{n+1}(\Z) \actson A_\Delta^{(n+1)}$
after appending with a row of zeros on the bottom.

\begin{proposition}\label{prop:adding_another_dimension}
	The map $f\colon  \Z^{n\times(m+1)} \to \Z^{(n+1)\times(m+1)}$
	which appends a row of zeros to a matrix from $\Z^{n\times(m+1)}$
	thereby making it an $(n+1)\times(m+1)$ matrix
	induces a map $\overline f\colon A_\Delta^{(n)} \to A_\Delta^{(n+1)}$.
	Furthermore, if $[A], [B] \in A_\Delta^{(n)}$ are such that
	$\hat\Phi\cdot [A] = [B]$ for some $\hat\Phi \in \GL_n(\widehat\Z)$,
	then there exists $\Phi \in \GL_{n+1}(\Z)$ such that $\Phi\cdot \overline f([A]) = f([B])$.
\end{proposition}

\begin{proof}
	Take $\Phi_A, \Phi_B$ \st\
	$[A'] = \Phi_A\cdot [A]$ and $[B'] = \Phi_B \cdot [B]$
	are upper-triangular as matrices
	in $\Z^n \oplus \bigoplus_{j=1}^{m-1}\big(\Z / d_j\big)^n$.
	Then let ${\hat\Phi' = \Phi_B\ \hat\Phi\ \Phi_A^{-1}}$
	so that $\hat\Phi' \cdot [A'] = [B']$.
	Take $\hat\Phi''$ to be the block diagonal matrix
	$$\hat\Phi'' =
	\begin{pmatrix}
		\hat\Phi		& \mathbf{0} \\
		\mathbf{0}^T	& (\det \hat\Phi)^{-1}
	\end{pmatrix}.$$
	Since the bottom rows of
	$\overline{f}([A'])$ and $\overline{f}([B'])$
	are both $\mathbf{0}^T$,
	we get the equality
	${\hat\Phi'' \cdot \overline{f}([A']) = \overline{f}([B'])}$.
	Also, $\det \hat\Phi'' = 1$,
	so we can choose a matrix ${\Phi \in \GL_{n+1}(\Z)}$
	such that ${\Phi \equiv \hat\Phi'' \mod d_{m-1}}$ 
	and such that
	if the first column of $\hat\Phi''$ is $\pm (1,0,\ldots, 0)^T$,
	then so is the first column of $\Phi$.
	As before, this implies that
	$\Phi \cdot \overline{f}([A']) = \overline{f}([B'])$ and so
	$$
	\begin{pmatrix}
		\Phi_B^{-1}	& \mathbf{0} \\
		\mathbf{0}^T	& 1
	\end{pmatrix}
	\cdot \Phi \cdot
	\begin{pmatrix}
		\Phi_A	& \mathbf{0} \\
		\mathbf{0}^T	& 1
	\end{pmatrix}\cdot \overline{f}([A]) = \overline{f}([B]).
	$$
\end{proof}

\subsection{Summary of results}

Now we are ready to summarise our state of knowledge.

\begin{proposition}\label{prop:summary}
	Let $n > 1$.
	Let $\Delta$ be a nice $2$-orbifold
	with cone points of orders $p_1, \ldots, p_m$.
	Let $d_1, \ldots, d_m$ be as in \cref{prop:Smith_form_for_A}
	-- in particular, $d_j\ |\ d_{j+1}$.
	Then, exactly one of the following statements holds.
	\begin{enumerate}
		\item $n > m$
		or $d_{m-(n-1)} \in \{1, 2, 3, 4, 6\}$,
		and
		$\hat\Phi\cdot [A] = [B]$ for some ${\hat\Phi \in \GL_n(\widehat\Z)}$
		implies that $\Phi \cdot [A] = [B]$ for some $\Phi \in \GL_n(\Z)$.
		\item $n \le m$
		and $d_{m-(n-1)} \not\in \{1, 2, 3, 4, 6, 12\}$,
		and there exist elements $[A], [B] \in A_\Delta^{(n)}$ such that
		$\hat\Phi\cdot [A] = [B]$ for some $\hat\Phi \in \GL_n(\widehat\Z)$,
		but there are no $\Phi \in \GL_n(\Z)$ and $\sigma\in\Sigma$
		such that $\Phi\cdot [A] = \sigma\cdot [B]$.
		\item $n \le m$ and $d_{m-(n-1)} = 12$.
	\end{enumerate}
\end{proposition}

\begin{proof}
		Following \cref{prop:profinite_matrix_multiplication_large_n}
		let $$k = 1 + |\big\{j \le m-1\ |\ d_j \not \in \{1, 2, 3, 4, 6\}\big\}|.$$
		If $n > m$, then clearly $k < n$.
		If $d_{m-(n-1)} \in \{1, 2, 3, 4, 6\}$, then
		$${d_1, \ldots, d_{m-(n-1)} \in \{1, 2, 3, 4, 6\}}$$
		and so
		$$\{j \le m-1\ |\ d_j \not \in \{1, 2, 3, 4, 6\}\} \sub \{m-(n-2), \ldots, m-1\},$$
		so $k \le |\{m-(n-2), \ldots, m-1\}| + 1 = n-1 < n$.
		Thus, the hypotheses of \cref{prop:profinite_matrix_multiplication_large_n} are satisfied
		so there indeed does exist $\Phi \in \GL_n(\Z)$ such that $\Phi\cdot[A] = [B]$.
		
		Alternatively, if $n\le m$ and $d_{m-(n-1)}\not\in\{1, 2, 3, 4, 6, 12\}$, then
		there is some prime power $p^\alpha \not \in \{2, 3, 4\}$ such that $p^\alpha\ |\ d_{m-(n-1)}$.
		Since $$d_{m-(n-1)}\ |\ d_{m-(n-2)}\ |\ \ldots\ |\ d_{m-1}\ |\ d_m,$$
		this implies that $p^\alpha$ divides at least $n$ of $p_1, \ldots, p_m$.
		By \cref{prop:examples_for_lack_of_rigidity},
		this implies that there are some $[A], [B] \in A_\Delta^{(n)}$ and $\hat\Phi \in \GL_n(\widehat\Z)$
		such that $[B] = \hat\Phi\cdot[A]$, but for no $\Phi\in\GL_n(\Z)$ and $\sigma\in\Sigma$
		do we have $\Phi\cdot [A] = \sigma \cdot [B]$.
		
		The only number which isn't in $\{1, 2, 3, 4, 6\}$
		and isn't divisible by a prime power different to $2, 3$ and $4$
		is $12$,
		so if conditions $(1)$ and $(2)$ don't hold,
		we must have $d_{m-(n-1)} = 12$.
\end{proof}

\begin{remark}
	The attentive reader certainly noticed that
	the `classification' isn't conclusive if ${d_{m-(n-1)} = 12}$.
	The reason for this is that in this case
	it is possible that
	there is no $\Phi \in \GL_n(\Z)$ such that
	$\Phi \cdot [A] = \hat \Phi \cdot [A] = [B]$,
	but there is a \emph{pair}
	$(\Phi, \sigma) \in \GL_n(\Z) \times \aut(\Delta)$
	such that $\Phi \cdot [A] = \sigma \cdot [B]$.
	The author spent considerable time
	trying to understand this case,
	but wasn't able to arrive at a full classification.
\end{remark}

\section{Distinguishing central extensions $\Z^n$-by-$\Delta$
	by their finite quotients}\label{sec:results}

Finally, having done the `calculations',
we can go on to proving our main result
-- deciding which central extensions
of \fg\ free abelian groups by infinite $2$-orbifold groups
are distinguished from each other by their finite quotients.

\Cref{ssec:kernel_and_quotient} reduces the question
to distinguishing between central extensions of the same $\Z^n$
by the same $2$-orbifold group $\Delta$. \Cref{ssec:final_distinguishing_extensions} solves this reduced question.

\subsection{Kernel and quotient must agree}\label{ssec:kernel_and_quotient}

\begingroup
\def\themainthm{\ref{thm:kernel_quotient_agree}}
\addtocounter{thmx}{-3}
\begin{thmx}
	Let $n_1, n_2$ be two natural numbers
	and $\Delta_1, \Delta_2$ be infinite
	fundamental groups
	of closed orientable $2$-orbifolds.
	Let $G_1$ and $G_2$ be central extensions
	$\Z^{n_1}$-by-$\Delta_1$
	and $\Z^{n_2}$-by-$\Delta_2$ respectively.
	If ${\widehat G_1 \cong \widehat G_2}$
	then $n_1 = n_2$ and $\Delta_1 \cong \Delta_2$.
\end{thmx}
\endgroup

\begin{proof}
	In \cref{cor:central_exts_are_rf} we showed
	that such central extensions are residually finite
	and that we get commutative diagrams of short exact sequences
	\begin{equation*}
		\begin{tikzcd}
			0 \arrow{r}	& \Z^{n_i} \arrow{r} \arrow{d}	& G_i \arrow{r}\arrow{d} & \Delta_i \arrow{r}\arrow{d} 	& 1 \\
			0 \arrow{r}	& \widehat\Z^{n_i} \arrow{r}	& \widehat G_i \arrow{r}	& \widehat \Delta_i \arrow{r}	& 1
		\end{tikzcd}
	\end{equation*}
	where all of the vertical maps are the natural inclusions.
	
	Firstly, let's exclude the case that $\Delta \cong \Z^2$
	and that $G_1$ is isomorphic to $\Z^{n_1+2}$.
	Then also $G_2 \cong \Z^{n_1+2}$
	and so $\Delta$ is abelian, and thus it is $\Z^2$.
	The kernel of a surjection from $\Z^{n_1+2}$ to $\Z^2$
	is isomorphic to $\Z^{n_1}$,
	so the theorem statement holds.
	
	In \cref{sssec:centres} we studied the centres
	of $2$-orbifold groups and their profinite completions,
	showing in \cref{thm:completions_nice_orbifolds_centerless}
	that for any nice $2$-orbifold group $\Delta$ we have $\mathcal{Z}(\widehat \Delta) = 1$.
	Since the extensions
	${0 \to \widehat\Z^{n_i} \to \widehat G_i \to \widehat \Delta_i \to 1}$ are also central,
	we get that $\mathcal{Z}(\widehat G_i) = \widehat\Z^{n_i}$.
	In \cref{prop:Z2} we showed that
	in the case of $\Delta \cong \Z^2$
	also $\mathcal{Z}(\widehat G_i) = \widehat \Z^{n_i}$
	unless $G_i \cong \Z^{n_i +2}$,
	which was covered in the previous paragraph.
	This gives
	\begin{align*}
		\widehat\Z^{n_1} \cong \mathcal{Z}(\widehat G_1)	& \cong \mathcal{Z}(\widehat G_2) \cong \widehat\Z^{n_2}, \\
		\widehat\Delta_1 \cong G_1 / \mathcal{Z}(\widehat G_1)	& \cong G_2 / \mathcal{Z}(\widehat G_2) \cong \widehat\Delta_2.
	\end{align*}
	Now, \fg\ abelian \gps\ are distinguished by their profinite completions, so we get $n_1 = n_2$.
	
	The isomorphism of $\Delta_1$ and $\Delta_2$
	follows directly from \cref{thm:orbifolds_rigidity}.
\end{proof}

\subsection{Distinguishing central extensions $\Z^n$-by-$\Delta$ fixing $n$ and $\Delta$}\label{ssec:final_distinguishing_extensions}

In this section, we state our main results
for distinguishing central extensions
of a fixed $\Z^n$
by a fixed infinite closed $2$-orbifold group $\Delta$.

\begingroup
\def\themainthm{\ref{thm:main}}
\begin{thmx}
	Let $\Delta$ be an infinite fundamental group
	of a closed orientable {$2$-orbifold}
	with $m \ge 0$ cone points
	of orders $p_1, \ldots, p_m$.
	Furthermore, for $j = 1, \ldots, m$ set
	$$d_j = \frac{\gcd(p_{i_1}p_{i_2}\ldots p_{i_j}\
		\text{for}\ 1\le i_1 < i_2 < \ldots < i_j \le m)}
	{\gcd(p_{i_1}p_{i_2}\ldots p_{i_{j-1}}\
		\text{for}\ 1\le i_1 < i_2 < \ldots < i_{j-1} \le m)}.$$
	Then, for $n > 1$ the following hold.
	\begin{enumerate}
		\item If $n > m$,
		or $d_{m-(n-1)} \in \{1, 2, 3, 4, 6\}$,
		then the non-isomorphic central extensions
		of $\Z^n$ by $\Delta$
		are distinguished from each other by their profinite completions.
		\item If $n \le m$
		and $d_{m-(n-1)} \not\in \{1, 2, 3, 4, 6, 12\}$,
		then there exist non-isomorphic central extensions $G_1, G_2$
		of $\Z^n$ by $\Delta$
		with $\widehat G_1 \cong \widehat G_2$.
	\end{enumerate}
\end{thmx}
\endgroup

\begin{proof}
	First, let's solve the case of $\Delta \cong \Z^2$.
	By \cref{prop:Z2}, the central extensions of $\Z^n$ by $\Z^2$
	are isomorphic to $H_k \times \Z^{n-1}$,
	where $H_k$ has presentation $\pres{a, b, c}{[a, c] = [b, c] = 1, [a, b] = c^k}$.
	The abelianisation of $H_k$ is
	${\Z^{n+1} \oplus (\Z/k)}$,
	so it is different for different values of $k$.
	Abelianisation is detected by profinite completions,
	so central extensions of $\Z^n$ by $\Z^2$
	are distinguished by their profinite completions.
	
	Now, let $\Delta \not \cong \Z^2$.
	For $j = 1, 2$ let
	$0 \to \Z^n \xrightarrow{\iota_j} G_j \xrightarrow{\pi_j} \Delta \to 1$
	be central extensions.
	From \cref{thm:matrix_correspondence}
	we know that $G_1 \cong G_2$ \Iff\
	the representatives $[A_j] \in A_{\Delta}^{(n)}$ lie
	in the same orbit of $\GL_n(\Z) \times \Sigma \actson A_\Delta^{(n)}$;
	furthermore, $\widehat G_1 \cong \widehat G_2$ \Iff\
	$\hat\Phi \cdot [A_1] = \sigma \cdot [A_2]$
	for some $\hat\Phi \in \GL_n(\widehat\Z)$ and $\sigma \in \Sigma$.
	
	\Cref{prop:summary} tells us when these conditions coincide
	given the hypotheses $(1)$ and $(2)$.
\end{proof}

The second main result shows
that all of examples of lack of rigidity that we produced
in fact come from being isomorphic
after a direct product with $\Z$.

\begingroup
\def\themainthm{\ref{thm:main_appendix}}
\begin{thmx}
	Let $\Delta$ be an infinite fundamental group
	of a closed orientable $2$-orbifold,
	with $m \ge 0$ cone points
	of orders $p_1, \ldots, p_m$.
	Let $n>1$ and $G_1, G_2$ be central extensions
	of $\Z^n$ by $\Delta$
	such that $\widehat G_1 \cong \widehat G_2$.
	Then $G_1 \times \Z \cong G_2 \times \Z$.
\end{thmx}
\endgroup

\begin{proof}
	Firstly, in the case of $\Delta \cong \Z^2$
	we have $\widehat G_1 \cong \widehat G_2 \iff G_1 \cong G_2$, so the statement holds.
	
	The short exact sequences
	$0 \to \Z^n \xrightarrow{\iota_j} G_j \xrightarrow{\pi_j} \Delta \to 1$
	stabilise to the sequences
	$$\begin{tikzcd}
		0 \arrow{r}
		& \Z^n \times \Z \arrow{r}{\iota_j \times \id_{\Z}}
		& G_j \times \Z \arrow{r}{\pi_j\times 1}
		& \Delta \arrow{r}
		& 1
	\end{tikzcd}$$
	represented by $\overline{f}\big([A_j]\big)$
	where $[A_j]\in A_\Delta^{(n)}$ represents the original extension,
	and where ${\overline{f}\colon  A_\Delta^{(n)} \to A_\Delta^{(n+1)}}$ is
	the map defined in \cref{prop:adding_another_dimension}.
	It is induced from ${f\colon  \Z^{n\times(m+1)} \to \Z^{(n+1)\times(m+1)}}$
	by appending an $n\times (m+1)$ matrix with a zero row at the bottom.
	
	Now, $\widehat G_1 \cong \widehat G_2$,
	so $\hat\Phi\cdot [A_1] = \sigma \cdot [A_2] = [\sigma \cdot A_2]$
	for some $\hat\Phi \in \GL_n(\widehat\Z)$ and $\sigma \in \Sigma$.
	In \cref{prop:adding_another_dimension} it was shown
	that $\hat\Phi\cdot [A] = [B]$ implies that
	$\Phi \cdot \overline{f}\big([A]\big) = \overline{f}\big([B]\big)$
	for some $\Phi \in \GL_{n+1}(\Z)$.
	In the present case, it means that
	$\Phi \cdot \overline f\big([A_1]\big)
	= \overline{f}\big([\sigma \cdot A_2]\big)
	= \sigma \cdot \overline{f}\big([A_2]\big)$
	and hence that $G_1 \times \Z \cong G_2 \times \Z$.
\end{proof}

%%%%%%%%%%%%%%%%%%%%%% REFERENCES HERE

\bibliographystyle{siam}
\bibliography{Bibliography}

\end{document}